\documentclass[11pt,reqno]{amsart}

\usepackage[a4paper,textwidth=6.5in,top=25mm,bottom=25mm,headheight=12pt,headsep=12pt]{geometry}
\usepackage{amsmath,amssymb,amsfonts,mathrsfs,mathtools}
\usepackage{enumitem}
\usepackage{array}
\usepackage{float}
\usepackage{microtype}
\usepackage{xcolor}
\usepackage{cite}
\makeatletter
\def\hyper@nopatch@thm{}
\makeatother
\usepackage[colorlinks=true,citecolor=blue,linkcolor=blue,urlcolor=black]{hyperref}

\allowdisplaybreaks[2]
\setlist{topsep=4pt plus 1pt minus 1pt,itemsep=2pt,parsep=0pt}
\numberwithin{equation}{section}

\newtheorem{theorem}{Theorem}[section]
\newtheorem{lemma}{Lemma}[section]
\newtheorem{proposition}[lemma]{Proposition}
\newtheorem{corollary}[theorem]{Corollary}
\theoremstyle{definition}
\newtheorem{definition}{Definition}[section]

\theoremstyle{remark}
\newtheorem{remark}{Remark}[section]

\title[Quasianalyticity and geometric rigidity]
{Quasianalyticity and geometric rigidity in the anisotropic Calder\'on problem}

\author{Lu Chen}
\address[Lu Chen]{Key Laboratory of Algebraic Lie Theory and Analysis, Ministry of Education, School of Mathematics and Statistics, Beijing Institute of Technology, Beijing 100081, PR China}
\email{chenlu5818804@163.com}

\author{Yan Jiang}
\address[Yan Jiang]{Department of Mathematics, City University of Hong Kong, Hong Kong SAR, China}
\email{yjian24@cityu.edu.hk}

\author{Hongyu Liu}
\address[Hongyu Liu]{Department of Mathematics, City University of Hong Kong, Hong Kong SAR, China}
\email{hongyu.liuip@gmail.com, hongyliu@cityu.edu.hk}

\author{Longyue Tao}
\address[Longyue Tao]{Department of Mathematics, City University of Hong Kong, Hong Kong SAR, China}
\email{sdyctly@163.com, longyue.tao@my.cityu.edu.hk}

\date{August 16, 2026}

\hypersetup{
  hypertexnames=false,
  pdfauthor={Lu Chen, Yan Jiang, Hongyu Liu, and Longyue Tao},
  pdftitle={Quasianalyticity and geometric rigidity in the anisotropic Calderon problem}
}

\begin{document}
\raggedbottom
\setlength{\abovedisplayskip}{7pt plus 2pt minus 2pt}
\setlength{\belowdisplayskip}{7pt plus 2pt minus 2pt}
\setlength{\abovedisplayshortskip}{0pt plus 2pt}
\setlength{\belowdisplayshortskip}{4pt plus 2pt minus 2pt}

\begin{abstract}
The anisotropic Calder\'on problem of determining a smooth Riemannian metric from boundary measurements, up to a boundary-fixing diffeomorphism, remains open in dimensions $n\ge3$~\cite{uhlmann2009electrical}.
We establish unique identifiability results in two complementary regimes.
In the first, the identity principle for quasianalytic functions propagates boundary information and yields unique identifiability on compact manifolds without a prescribed product structure, including a partial-data consequence; under a prescribed normal geometry, quasianalyticity is needed only in the distinguished direction.
In the second, suitable symmetry or one-sided ordering assumptions lead to unique identifiability at $C^\infty$ regularity with full or restricted boundary access.
Taken together, the results exhibit a tradeoff among regularity, geometric structure, and boundary access: quasianalyticity supplies continuation when no global product structure is prescribed, while symmetry or one-sided order replaces that continuation at $C^\infty$ regularity.
\end{abstract}

\maketitle

\section{Introduction}\label{sec:introduction}

Let $M$ be a compact connected smooth $n$-manifold ($n\geq 3$) with boundary, and let $g$ be a smooth Riemannian metric on $M$.
For $f\in C^\infty(\partial M)$, let $u_g^f$ be the solution of
\begin{equation}
  \Delta_g u_g^f=0\quad\text{in }M,
  \qquad u_g^f|_{\partial M}=f.
  \label{eq:2}
\end{equation}
Here $\Delta_g$ is the Laplace--Beltrami operator, given in local coordinates by
\[
  \Delta_g u
  =|g|^{-1/2}\partial_a\bigl(|g|^{1/2}g^{ab}\partial_bu\bigr),
  \qquad (g^{ab})=(g_{ab})^{-1},
  \qquad |g|=\det(g_{ab}).
\]
Writing $\nu_g$ for the outward $g$-unit normal along $\partial M$, define the Dirichlet-to-Neumann (DtN) map by
\begin{equation}
  \Lambda_g f=\partial_{\nu_g}u_g^f|_{\partial M}.
  \label{eq:geometric-dn-main}
\end{equation}
The classical Calder\'on problem concerns determining $(M, g)$ from knowledge of $\Lambda_g$. The issue of unique identifiability is of primary importance: given any two such metrics $g_1$ and $g_2$, one asks whether equality of the boundary measurements implies the existence of a diffeomorphism $F:M\to M$ such that
\begin{equation}
  \Lambda_{g_1}=\Lambda_{g_2}
  \quad\Longrightarrow\quad
  F^*g_2=g_1,
  \qquad F|_{\partial M}=\operatorname{Id}.
  \label{eq:anisotropic-inverse-problem}
\end{equation}
The diffeomorphism in this conclusion is unavoidable: if $H:M\to M$ satisfies $H|_{\partial M}=\operatorname{Id}$, then $\Lambda_{H^*g}=\Lambda_g$~\cite{sylvester1990anisotropic}.
The condition $H|_{\partial M}=\operatorname{Id}$ fixes boundary points but does not force $DH$ to be the identity in the normal direction.
Thus an a priori coordinate gauge is needed if one seeks literal equality of the metrics.
Calder\'on's foundational formulation~\cite{calderon1980inverse} initiated the modern inverse conductivity problem, and the global uniqueness theorem of Sylvester and Uhlmann~\cite{sylvester1987global} established the higher-dimensional isotropic result, where $g=\alpha e$, with $\alpha$ a scalar function and $e$ the Euclidean metric. The anisotropic problem, in which $g$ is a full Riemannian metric, is much subtler precisely because of the diffeomorphism gauge; its general smooth form in dimensions $n\ge3$ remains a longstanding open problem~\cite{uhlmann2009electrical}.
In dimension two, the anisotropic Calder\'on problem has a separate quasiconformal theory, including an application involving two sets of local boundary data~\cite{astala2005calderon}.

In this paper, we establish unique identifiability of metrics in suitable quasianalytic Denjoy--Carleman classes on compact manifolds without a prescribed product structure, from full boundary data or measurements on an arbitrary nonempty open boundary portion, with the corresponding diffeomorphism gauge.
For metrics in a prescribed global boundary normal form, we prove that quasianalyticity in the normal variable alone suffices for unique identifiability.
We also establish unique identifiability at $C^\infty$ regularity for metrics in a prescribed product form under tangential homogeneity with full boundary data, and for anisotropic conductivities under layerwise Loewner ordering with partial boundary data.

A classical route to global rigidity combines boundary determination with analytic continuation.
For smooth metrics in boundary normal coordinates, when boundary points are compared by the identity map, the full symbol of the DtN map determines the full Taylor series of the metric at the boundary.
Lee and Uhlmann used real analyticity to propagate this boundary identification into the interior and obtain uniqueness up to the natural gauge~\cite{lee1989determining}.
Lassas and Uhlmann subsequently removed the earlier topological restrictions and allowed measurements on an arbitrary nonempty open boundary portion in the real-analytic setting~\cite{lassas2001determining}.
The boundary determination component was developed first for conductivities and then for anisotropic conductivities and related geometric operators~\cite{kohn1984determining,kohn1985determining,sylvester1988inverse,nakamura2001local,kang2002boundary,joshi2005inverse}.
Related reconstruction results were obtained for complete Riemannian manifolds with boundary~\cite{lassas2003dirichlet}.
More recently, the Poisson embedding method supplied a geometric global candidate for the boundary-fixing isometry and reorganized the continuation argument through harmonic functions and Runge approximation~\cite{lassas2020poisson}.
Its formulation using boundary values supported in an arbitrary open set $\Gamma$ is the starting point for the partial-data corollary below.
These works show that analytic regularity provides a powerful mechanism for passing from boundary information to a global geometric identification.
Nonuniqueness has been established in every nonanalytic Gevrey class for related inverse problems with a fixed nonconstant potential or at a fixed nonzero frequency~\cite{daude2026sharp}.
These results concern different equations from the zero-frequency problem considered here.

There is a complementary line in which special geometric structure supplies the rigidity needed for global recovery.
Einstein metrics can be recovered from boundary Cauchy data by exploiting the elliptic structure of the Einstein equation~\cite{guillarmou2009inverse}.
On conformally transversally anisotropic manifolds, complex geometrical optics and Gaussian beam constructions relate uniqueness, linearized injectivity, and reconstruction to geodesic ray transforms on the transversal manifold~\cite{dos2009limiting,dos2016calderon,dos2020linearized,feizmohammadi2021reconstruction}.
Partial-data results are also available on manifolds with suitable anisotropic geometry~\cite{kenig2013calderon}.
In a prescribed transversally anisotropic product geometry, partial Neumann-to-Dirichlet data can also determine both the transversal metric and a potential through reductions to nonlocal elliptic and wave inverse problems~\cite{lin2024calderon}.
Separation of variables and inverse spectral theory give further anisotropic results for warped-product and conformally St\"ackel geometries~\cite{daude2020anisotropic,daude2021anisotropic}.
This raises the question of how much regularity is needed for uniqueness in a prescribed global boundary normal form.
We prove that quasianalyticity in the normal variable suffices, while tangential homogeneity permits uniqueness at $C^\infty$ regularity in the product setting considered below.

Partial boundary measurements are especially delicate for anisotropic coefficients because the restricted measurements retain a diffeomorphism gauge tied to the observed boundary portion.
Unless otherwise specified, partial data on $\Gamma$ means that Dirichlet inputs are supported in a nonempty relatively open set $\Gamma\subset\partial M$ and the corresponding Neumann outputs are observed on that same set.
The natural gauge therefore fixes $\Gamma$, not necessarily the inaccessible boundary.
Foundational partial-data results for Schr\"odinger potentials and isotropic conductivities include~\cite{bukhgeim2002recovering,kenig2007calderon,isakov2007uniqueness,imanuvilov2010calderon,kenig2013calderon}.
Relevant anisotropic precedents include the real-analytic Riemannian and two-dimensional elliptic theories~\cite{lassas2001determining,imanuvilov2011determination}.
Other results concern local recovery, piecewise or layered matrix conductivities, and uniqueness under definiteness assumptions~\cite{alessandrini2009local,alessandrini2017uniqueness,alessandrini2018eit,garde2025direct}.
Cross data allow different input and observation sets, with disjoint sets as a special case~\cite{imanuvilov2011inverse}.
We establish uniqueness under quasianalytic regularity or layerwise Loewner ordering, with both inputs and observations confined to the same arbitrary nonempty open boundary portion.

Monotonicity and localized potentials have been used for shape reconstruction in electrical impedance tomography~\cite{harrach2013monotonicity} and for local uniqueness of uniformly positive Schr\"odinger potentials from partial boundary data under a local one-sided ordering~\cite{harrach2017local}.
Such results indicate that order conditions can support identifiability without an analyticity assumption.
For smooth anisotropic conductivities in fixed coordinates, we prove global uniqueness from partial boundary data under layerwise one-sided Loewner ordering and the prescribed foliation hypotheses.

The quasianalytic Denjoy--Carleman classes considered here retain the identity principle that a function is determined by its jet at one point, without requiring representation by a convergent power series.
The additional stability assumptions stated below provide the local calculus and elliptic regularity needed for the geometric argument~\cite{komatsu1979implicit,kriegl2011convenient,rainer2014composition,furdos2020geometric,thilliez2008quasianalytic}.
Quasianalyticity has also been used to propagate partial DtN data to full DtN data~\cite{daude2026local}.
That result propagates boundary spectral information.
Here, by contrast, quasianalyticity of the interior metric supplies the continuation needed to pass from boundary determination to a global geometric identification.

\subsection{Main results}\label{sec:main-results}

We now state the four uniqueness results.

\subsubsection{Quasianalyticity on general manifolds}

We first specify the weight sequence and the corresponding Denjoy--Carleman class of Roumieu type used in Theorem~\ref{thm:global-quasianalytic}; for the standard terminology and conventions for weight sequences, see~\cite{kriegl2011convenient,furdos2020geometric}.

\begin{definition}\label{def:admissible-quasianalytic-weight}
Let $\mathbf{W}=(W_k)_{k\ge0}$ be a sequence of positive numbers.
We say that $\mathbf{W}$ is an \emph{admissible quasianalytic weight sequence} if
\begin{align*}
  &W_0=W_1=1,
  \qquad \sup_{k\ge1}\left(\frac{W_{k+1}}{(k+1)W_k}\right)^{1/k}<\infty,
  \qquad
  \left(\frac{W_k}{k!}\right)^2
  \le
  \frac{W_{k-1}}{(k-1)!}\frac{W_{k+1}}{(k+1)!}
  \quad(k\ge1), \tag{A1}\label{eq:A1}\\
  &\lim_{k\to\infty}\left(\frac{W_k}{k!}\right)^{1/k}=\infty, \tag{A2}\label{eq:A2}\\
  &W_{j+k}\le C_0R_0^{j+k}W_jW_k\qquad(j,k\ge0) \tag{A3}\label{eq:A3}
\end{align*}
for some $C_0,R_0>0$, together with the Denjoy--Carleman quasianalyticity condition
\begin{equation*}
  \sum_{k=1}^\infty\frac{W_{k-1}}{W_k}=\infty.
  \tag{A4}\label{eq:A4}
\end{equation*}
\end{definition}

\begin{definition}\label{def:roumieu-denjoy-carleman-class}
Let $\mathbf{W}$ be an admissible quasianalytic weight sequence in the sense of Definition~\ref{def:admissible-quasianalytic-weight}.
For each integer $d\ge1$ and each open set $U\subset\mathbb{R}^d$, define the \emph{Denjoy--Carleman class of Roumieu type} $\mathcal{E}^{\{\mathbf{W}\}}(U)$ to consist of all functions $a\in C^\infty(U)$ with the following property: for every $K\Subset U$, there exist constants $C_K,R_K>0$ such that
\begin{equation}
  \sup_{x\in K}|\partial^\alpha a(x)|
  \le C_KR_K^{|\alpha|}W_{|\alpha|}
  \qquad\text{for every multiindex }\alpha.
  \label{eq:1}
\end{equation}
\end{definition}

We use $\mathcal{E}^{\{\mathbf{W}\}}$ in the chartwise sense and refer to it as the \emph{quasianalytic class associated with $\mathbf{W}$}.
A smooth manifold is of class $\mathcal{E}^{\{\mathbf{W}\}}$ if it admits an atlas whose transition maps belong to $\mathcal{E}^{\{\mathbf{W}\}}$, and maps and tensor fields are of this class when their coordinate representations are.
For manifolds with boundary, the transition maps in boundary charts and the coordinate coefficients of tensor fields are required to admit $\mathcal{E}^{\{\mathbf{W}\}}$ extensions across the boundary hyperplane.
An $\mathcal{E}^{\{\mathbf{W}\}}$ diffeomorphism is a diffeomorphism whose map and inverse are both of this class.

\begin{theorem}\label{thm:global-quasianalytic}
Let $\mathbf{W}$ be an admissible quasianalytic weight sequence as in Definition~\ref{def:admissible-quasianalytic-weight}, and write $\mathcal{E}^{\{\mathbf{W}\}}$ for the associated quasianalytic class defined in Definition~\ref{def:roumieu-denjoy-carleman-class}.
Let $n\ge3$, and let $M$ be a compact connected $n$-manifold of class $\mathcal{E}^{\{\mathbf{W}\}}$ with nonempty boundary.
Let $g_1$ and $g_2$ be Riemannian metrics of class $\mathcal{E}^{\{\mathbf{W}\}}$ on $M$.
If their DtN maps agree,
\begin{equation}
  \Lambda_{g_1}=\Lambda_{g_2},
  \label{eq:3}
\end{equation}
then there exists an $\mathcal{E}^{\{\mathbf{W}\}}$ diffeomorphism $F:M\to M$ such that
\begin{equation}
  F|_{\partial M}=\operatorname{Id},
  \qquad F^*g_2=g_1.
  \label{eq:4}
\end{equation}
\end{theorem}

For the partial-data result, fix a nonempty relatively open subset $\Gamma\subset\partial M$, and set $M^\circ=M\setminus\partial M$ and $M^\Gamma=M^\circ\cup\Gamma$.
For a piecewise $C^1$ curve $c:[0,1]\to M$, set $L_g(c):=\int_0^1|\dot c(t)|_g\,\mathrm{d}t$.
For $x,y\in M$, define the Riemannian length distance by
\begin{equation}
  d_g(x,y):=\inf_c L_g(c),
  \label{eq:riemannian-length-distance}
\end{equation}
where the infimum is taken over all such curves joining $x$ to $y$.
For $f\in C_c^\infty(\Gamma)$, let $\widetilde f$ denote its extension by zero to $\partial M$.
The partial DtN map $\Lambda_g^\Gamma:C_c^\infty(\Gamma)\to C^\infty(\Gamma)$ is defined by
\begin{equation}
  \Lambda_g^\Gamma f=\bigl(\Lambda_g\widetilde f\bigr)|_\Gamma.
  \label{eq:local-quasianalytic-dn}
\end{equation}

\begin{corollary}\label{cor:global-quasianalytic-local}
Let $\mathbf{W}$ and the associated quasianalytic class $\mathcal{E}^{\{\mathbf{W}\}}$ be as in Theorem~\ref{thm:global-quasianalytic}.
Let $n\ge3$, and let $M$ be a compact connected $n$-manifold of class $\mathcal{E}^{\{\mathbf{W}\}}$ with nonempty boundary.
Let $g_1$ and $g_2$ be Riemannian metrics of class $\mathcal{E}^{\{\mathbf{W}\}}$ on $M$.
Let $\Gamma\subset\partial M$ be nonempty and relatively open.
If the partial DtN maps defined in \eqref{eq:local-quasianalytic-dn} agree,
\begin{equation}
  \Lambda_{g_1}^\Gamma=\Lambda_{g_2}^\Gamma
  \label{eq:local-quasianalytic-data}
\end{equation}
then there exists an $\mathcal{E}^{\{\mathbf{W}\}}$ diffeomorphism $F:M^\Gamma\to M^\Gamma$ such that
\begin{equation}
  F|_\Gamma=\operatorname{Id},
  \qquad F^*g_2=g_1\quad\text{on }M^\Gamma.
  \label{eq:local-quasianalytic-conclusion}
\end{equation}
Moreover, the isometry $F|_{M^\circ}$ extends uniquely to an isometry between the metric completions of $(M^\circ,d_{g_1})$ and $(M^\circ,d_{g_2})$.
Under the canonical identifications of these completions with $(M,d_{g_1})$ and $(M,d_{g_2})$, respectively, the extension defines an isometry $\overline F:(M,d_{g_1})\to(M,d_{g_2})$.
Moreover,
\[
  \overline F|_\Gamma=\operatorname{Id},
  \qquad
  \overline F(\partial M\setminus\Gamma)=\partial M\setminus\Gamma.
\]
\end{corollary}

The condition $F|_\Gamma=\operatorname{Id}$ is the natural gauge for the partial-data problem; the corollary does not assert that $\overline F$ fixes the entire boundary pointwise.
Indeed, if $H:M\to M$ is a smooth diffeomorphism satisfying $H|_\Gamma=\operatorname{Id}$, then $g$ and $H^*g$ have the same partial DtN map on $\Gamma$.
To use this invariance within the quasianalytic class considered here, the pullback metric $H^*g$ must also belong to $\mathcal E^{\{\mathbf W\}}$.

\begin{remark}\label{rem:analytic-to-quasianalytic}
Boundary determination is a smooth argument: it recovers the full boundary jet of the metric in boundary normal coordinates.
Quasianalyticity enters only when equality of these jets is used to obtain a local isometry near the boundary and when that isometry is extended to the interior using the Poisson embedding.
Corollary~\ref{cor:global-quasianalytic-local} establishes uniqueness from partial boundary data, including the unique extension of the resulting interior isometry to the metric completions.
\end{remark}

\begin{remark}\label{rem:quasianalytic-metric-example}
Theorem~\ref{thm:global-quasianalytic} applies to metrics whose coefficients need not be real analytic in a fixed real-analytic atlas.
For admissible quasianalytic Denjoy--Carleman classes strictly larger than the real-analytic class, see~\cite[Section~2]{furdos2020geometric}.
\end{remark}

\subsubsection{Quasianalyticity in the normal variable}
We next consider the anisotropic Calder\'on problem \eqref{eq:2}--\eqref{eq:anisotropic-inverse-problem} in a prescribed global boundary normal coordinate system, imposing quasianalyticity only in the normal variable.
Let $n\ge3$ be the dimension of the manifold, let $L>0$, let $Y$ be a closed connected smooth $(n-1)$-manifold, and set
\begin{equation}
  M=[0,L]\times Y,
  \qquad \Gamma_0=\{0\}\times Y,
  \qquad \Gamma_L=\{L\}\times Y.
  \label{eq:9}
\end{equation}
The coordinate $t$ is the prescribed distance from $\Gamma_0$, while $y$ labels the normal geodesic issuing from $(0,y)$.
The boundary data are the full DtN map defined in \eqref{eq:geometric-dn-main}, compared under the identity identification of $\partial M=\Gamma_0\sqcup\Gamma_L$.
We first isolate the weaker assumptions on the weight sequence needed for continuation along the normal curves.

\begin{definition}\label{def:one-dimensional-quasianalytic-weight}
Let $\mathbf{W}=(W_k)_{k\ge0}$ be a sequence of positive numbers.
We say that $\mathbf W$ is a \emph{one-dimensional quasianalytic weight sequence} if $W_0=1$, it is logarithmically convex,
\begin{equation*}
  W_k^2\le W_{k-1}W_{k+1}
  \qquad(k\ge1),
\end{equation*}
and it satisfies the Denjoy--Carleman quasianalyticity condition
\begin{equation}
  \sum_{k=1}^\infty\frac{W_{k-1}}{W_k}=\infty.
  \label{eq:10}
\end{equation}
\end{definition}

Every weight sequence admissible in Definition~\ref{def:admissible-quasianalytic-weight} satisfies Definition~\ref{def:one-dimensional-quasianalytic-weight}, whereas the latter omits the additional assumptions \eqref{eq:A1}--\eqref{eq:A3} needed for multivariable quasianalytic calculus and elliptic regularity.
Here this weaker condition is imposed only on the normal dependence of the metric coefficients.

\begin{definition}\label{def:normal-variable-quasianalytic-class}
Let $\mathbf{W}$ be a one-dimensional quasianalytic weight sequence in the sense of Definition~\ref{def:one-dimensional-quasianalytic-weight}.
Relative to the fixed product structure $M=[0,L]\times Y$, define $\mathcal Q_{\mathbf W}(M)$ to be the set of smooth Riemannian metrics of the form
\begin{equation}
  g=\,\mathrm{d} t^2+h(t,y),
  \label{eq:11}
\end{equation}
where $h(t,\cdot)$ is a Riemannian metric on $Y$ for every $t\in[0,L]$, and the family $h$ admits a smooth extension $\widetilde h$ to $(-\varepsilon,L+\varepsilon)\times Y$ for some $\varepsilon>0$ with the following property.
For every coordinate chart $(U,y)$ on $Y$, every $K\Subset U$, and every compact interval $I_0\Subset(-\varepsilon,L+\varepsilon)$, there exist constants $C_{I_0,K},R_{I_0,K}>0$ such that
\begin{equation}
  \sup_{(t,y)\in I_0\times K}
  \big|\partial_t^k \widetilde h_{\mu\nu}(t,y)\big|
  \le C_{I_0,K}R_{I_0,K}^kW_k
  \qquad(k\ge0,\ 1\le\mu,\nu\le n-1).
  \label{eq:12}
\end{equation}
We call $\mathcal Q_{\mathbf W}(M)$ the \emph{normal-variable quasianalytic metric class} associated with $\mathbf W$.
\end{definition}

\begin{theorem}\label{thm:normal-quasianalytic}
Let $\mathbf W$ be a one-dimensional quasianalytic weight sequence as in Definition~\ref{def:one-dimensional-quasianalytic-weight}.
Let $n\ge3$, let $L>0$, let $Y$ be a closed connected smooth $(n-1)$-manifold, and set $M=[0,L]\times Y$, with boundary components $\Gamma_0$ and $\Gamma_L$ as in \eqref{eq:9}.
Let $g_1,g_2\in\mathcal Q_{\mathbf W}(M)$, where $\mathcal Q_{\mathbf W}(M)$ is the normal-variable quasianalytic metric class defined in Definition~\ref{def:normal-variable-quasianalytic-class}, and write $g_j=\,\mathrm{d} t^2+h_j(t,y)$ for $j=1,2$.
If their DtN maps agree,
\begin{equation}
  \Lambda_{g_1}=\Lambda_{g_2}
  \label{eq:13}
\end{equation}
then
\begin{equation}
  g_1=g_2\qquad\text{on }M,
  \label{eq:14}
\end{equation}
or equivalently $h_1(t,y)=h_2(t,y)$ on $[0,L]\times Y$.
\end{theorem}

The assumption~\eqref{eq:13} may be replaced, for either $\Gamma\in\{\Gamma_0,\Gamma_L\}$, by
\[
  (\Lambda_{g_1}f)|_\Gamma=(\Lambda_{g_2}f)|_\Gamma
  \qquad\text{for every }f\in C^\infty(\partial M).
\]

\begin{corollary}\label{cor:product-base}
Let $\mathbf W$ be admissible as in Definition~\ref{def:admissible-quasianalytic-weight}, and let $M=M_1\times M_2$ have dimension $n\ge3$, where $M_1$ is a compact connected $\mathcal E^{\{\mathbf W\}}$ manifold with nonempty boundary and $M_2$ is a closed connected smooth manifold.
Fix a nonempty relatively open set $\Sigma\subset\partial M_1$ and set $A=\Sigma\times M_2$.
In the prescribed product coordinates, let
\[
  g_j=k(x)+h_j(x,y),\qquad j=1,2,
\]
with no mixed terms, where $x\in M_1$, $y\in M_2$, $k$ is a fixed known smooth metric on $M_1$, and $h_j(x,\cdot)$ is a smooth family of metrics on $M_2$.
Assume the coefficients of $h_j$ satisfy~\eqref{eq:1} with derivatives in all $x$ variables, locally uniformly in $y$, also in smooth local extensions across $\partial M_1$.
If the DtN observations agree on the whole set $A$, namely,
\[
  (\Lambda_{g_1}f)|_A=(\Lambda_{g_2}f)|_A
  \qquad\text{for every }f\in C^\infty(\partial M),
\]
then $g_1=g_2$ on $M$.
\end{corollary}

\begin{remark}\label{rem:normal-anisotropic-regularity}
Estimate \eqref{eq:12} means that the metric coefficients are quasianalytic in the normal variable $t$ but only $C^\infty$ in the tangential variables $y$.
Indeed, the same constants $C_{I_0,K}$ and $R_{I_0,K}$ control all normal derivative orders $k\ge0$, locally uniformly in $y$.
No growth condition is imposed on tangential derivatives, so the dependence on $y$ is only required to be smooth.
\end{remark}

\begin{remark}\label{rem:normal-geometry}
Theorem~\ref{thm:global-quasianalytic} treats a general manifold, where the boundary data do not determine a preferred correspondence between interior points; its proof uses a Poisson embedding to construct such a correspondence and yields uniqueness up to a boundary-fixing diffeomorphism.
In Theorem~\ref{thm:normal-quasianalytic}, the prescribed global boundary normal coordinates already identify the same point $(t,y)$ for both metrics, so no Poisson embedding or unknown interior diffeomorphism is needed.
Once the boundary jets are recovered, the one-dimensional identity principle along $t\mapsto(t,y)$ gives the literal equality \eqref{eq:14} in these coordinates.
The price is the stronger geometric assumption that one prescribed boundary normal product chart covers the whole manifold.
\end{remark}

\subsubsection{Tangential homogeneity}
The next result establishes uniqueness at $C^\infty$ regularity under tangential homogeneity.
We retain the global boundary normal form of Theorem~\ref{thm:normal-quasianalytic}, but impose tangential translation invariance and allow arbitrary smooth dependence on the normal variable.

Let $d\ge2$ and $L>0$, and let $\mathbb{T}^d=\mathbb{R}^d/(2\pi\mathbb{Z})^d$ denote the standard $d$-dimensional torus.
Set
$M=[0,L]\times\mathbb{T}^d,$
so that $M$ has dimension $d+1$.
We fix the product coordinates $(t,y)$ and compare the full DtN maps defined in~\eqref{eq:geometric-dn-main} under the identity identification of $\partial M$.
Let $\operatorname{Sym}_d$ and $\operatorname{Sym}_d^+$ denote the real symmetric and real symmetric positive definite $d\times d$ matrices, respectively.

\begin{definition}\label{def:tangentially-homogeneous-class}
Relative to the prescribed product coordinates on $M=[0,L]\times\mathbb{T}^d$, define the \emph{tangentially homogeneous metric class} $\mathcal S(M)$ to consist of all Riemannian metrics of the form
\begin{equation}
  g=\,\mathrm{d} t^2+\sum_{\alpha,\beta=1}^{d}h_{\alpha\beta}(t)\,\mathrm{d} y^\alpha\,\mathrm{d} y^\beta,
  \label{eq:15}
\end{equation}
where $h=(h_{\alpha\beta})\in C^\infty([0,L];\operatorname{Sym}_d^+)$.
\end{definition}

\begin{theorem}\label{thm:tangentially-homogeneous}
Let $d\ge2$, let $L>0$, and set $M=[0,L]\times\mathbb{T}^d$, with the prescribed product coordinates $(t,y)$.
Let $g_1,g_2\in\mathcal S(M)$, where $\mathcal S(M)$ is the tangentially homogeneous metric class defined in Definition~\ref{def:tangentially-homogeneous-class}.
If their DtN maps agree,
\begin{equation}
  \Lambda_{g_1}=\Lambda_{g_2}
  \label{eq:tangentially-homogeneous-data}
\end{equation}
then
\begin{equation}
  g_1=g_2\qquad\text{on }M.
  \label{eq:tangentially-homogeneous-conclusion}
\end{equation}
\end{theorem}

\begin{remark}\label{rem:homogeneous-smooth}
The equality $g_1=g_2$ in Theorem~\ref{thm:tangentially-homogeneous} is obtained within the prescribed global product form~\eqref{eq:15}.
\end{remark}

\subsubsection{Layerwise Loewner ordering}
We next consider partial-data uniqueness for smooth anisotropic conductivities under a layerwise ordering assumption.
Let $n\ge2$, let $\Omega\subset\mathbb{R}^n$ be a bounded connected smooth domain admitting an admissible layer function in the sense of Definition~\ref{def:admissible-layer-foliation}, and let $\Gamma\subset\partial\Omega$ be nonempty and relatively open.
We work in fixed Euclidean coordinates and write $\gamma_g=\sqrt{\det g}\,g^{-1}$ for the conductivity density induced by a Riemannian metric $g$.
For a smooth uniformly positive conductivity $\gamma$, let $\Lambda_\gamma^\Gamma$ denote the partial DtN map for $-\operatorname{div}(\gamma\nabla)$, with conormal flux as output and the boundary restriction in~\eqref{eq:local-quasianalytic-dn}.

The next two definitions specify the layer geometry and the ordering of conductivity pairs.

\begin{definition}\label{def:admissible-layer-foliation}
Let $T>0$ and $\rho\in C^\infty(\overline\Omega;[0,T])$.
For $0<s<T$, set $D_s=\{x\in\Omega:\rho(x)>s\}$.
We say that $\rho$ is an \emph{admissible layer function} if
\begin{align}
  &\rho^{-1}(0)=\partial\Omega,
  \qquad \,\mathrm{d}\rho\ne0\quad\text{on }\{0\le\rho<T\},
  \label{eq:loewner-foliation}\\
  &\begin{aligned}
    &D_s\Subset\Omega,
    \qquad \Omega\setminus\overline{D_s}\text{ is connected}
    \qquad(0<s<T),\\
    &\operatorname{int}_{\Omega}\{\rho=T\}=\varnothing.
  \end{aligned}
  \label{eq:loewner-topology}
\end{align}
\end{definition}

\begin{definition}\label{def:layerwise-loewner-class}
Fix an admissible layer function $\rho$ in the sense of Definition~\ref{def:admissible-layer-foliation}.
The \emph{layerwise one-sided Loewner class} $\mathcal L_\rho(\Omega)$ consists of pairs $(\gamma_1,\gamma_2)$ of uniformly positive conductivities in $C^\infty(\overline\Omega;\operatorname{Sym}_n^+)$ satisfying the following comparability condition in the prescribed Euclidean coordinates.
More precisely, for every $s\in[0,T)$, there exist an upper level $b_s\in(s,T]$ and a sign $\sigma_s\in\{1,-1\}$, which may depend on $s$ but is fixed throughout the entire layer $s<\rho(x)<b_s$, such that
\begin{equation}
  \sigma_s\bigl(\gamma_1(x)-\gamma_2(x)\bigr)\succeq0
  \qquad\text{whenever }s<\rho(x)<b_s,
  \label{eq:loewner-order}
\end{equation}
where $\succeq0$ denotes positive semidefiniteness.
\end{definition}

\begin{theorem}\label{thm:loewner}
Let $n\ge2$, and let $\Omega\subset\mathbb{R}^n$ be a bounded connected smooth domain.
Let $\Gamma\subset\partial\Omega$ be nonempty and relatively open.
Let $\rho$ be an admissible layer function as in Definition~\ref{def:admissible-layer-foliation}.
Let $(\gamma_1,\gamma_2)\in\mathcal L_\rho(\Omega)$, where $\mathcal L_\rho(\Omega)$ is the layerwise one-sided Loewner class defined in Definition~\ref{def:layerwise-loewner-class}.
If their partial DtN maps agree,
\begin{equation}
  \Lambda_{\gamma_1}^\Gamma=\Lambda_{\gamma_2}^\Gamma
  \label{eq:loewner-data}
\end{equation}
then
\begin{equation}
  \gamma_1=\gamma_2\qquad\text{on }\Omega.
  \label{eq:loewner-conclusion}
\end{equation}
\end{theorem}

The conductivity result also yields uniqueness for Riemannian metrics.

\begin{corollary}\label{cor:loewner-metric}
Let $n\ge3$, and let $\Omega$, $\Gamma$, and $\rho$ be as in Theorem~\ref{thm:loewner}.
Let $g_1,g_2$ be smooth Riemannian metrics on $\overline\Omega$.
Let $F:\overline\Omega\to\overline\Omega$ be a smooth diffeomorphism with $F|_\Gamma=\operatorname{Id}$.
Assume that, for this $F$, the pair $(\gamma_{g_1},\gamma_{F^*g_2})$ satisfies the comparability condition defining $\mathcal L_\rho(\Omega)$ in Definition~\ref{def:layerwise-loewner-class}.
If their partial conductivity DtN maps agree,
\begin{equation}
  \Lambda_{\gamma_{g_1}}^\Gamma=\Lambda_{\gamma_{g_2}}^\Gamma
  \label{eq:loewner-metric-data}
\end{equation}
then
\begin{equation}
  g_1=F^*g_2\qquad\text{on }\Omega.
  \label{eq:loewner-metric-conclusion}
\end{equation}
In particular, if $(\gamma_{g_1},\gamma_{g_2})\in\mathcal L_\rho(\Omega)$, then $g_1=g_2$.
\end{corollary}

\begin{remark}\label{rem:metric-conductivity}
For $n\ge3$, the metric and anisotropic conductivity formulations are equivalent under the standard correspondence $g\mapsto\gamma_g$, whose inverse is given by~\eqref{eq:metric-from-density}.
The conductivity DtN map is obtained by multiplying the output of the geometric DtN map by the Riemannian boundary surface density relative to the fixed reference boundary measure; this density is locally determined by the principal symbol of the conductivity DtN map.
Theorems~\ref{thm:global-quasianalytic}, \ref{thm:normal-quasianalytic}, and~\ref{thm:tangentially-homogeneous} and Corollary~\ref{cor:global-quasianalytic-local} therefore have conductivity versions under the corresponding assumptions, which we do not restate.
We use the conductivity formulation in Theorem~\ref{thm:loewner} because the layerwise ordering and the monotonicity argument are naturally expressed in terms of the conductivity difference.
\end{remark}

\subsection{Summary and organization}\label{subsec:summary-organization}

These results show how regularity and geometric structure provide distinct mechanisms for geometric rigidity: quasianalyticity supplies interior continuation, while prescribed structure can reduce or replace its role.
They concern different combinations of assumptions, not a nested hierarchy of uniqueness results.
Table~\ref{tab:main-results} summarizes the regularity, geometric structure, and boundary data in each setting.

\begin{table}[!htbp]
\centering
\small
\renewcommand{\arraystretch}{1.08}
\begin{tabular}{
|>{\raggedright\arraybackslash}p{0.23\textwidth}
|>{\raggedright\arraybackslash}p{0.27\textwidth}
|>{\raggedright\arraybackslash}p{0.40\textwidth}|}
\hline
\textbf{Result} & \textbf{Regularity} & \textbf{Geometry / structure / data} \\
\hline
Thm.~\ref{thm:global-quasianalytic} and Cor.~\ref{cor:global-quasianalytic-local}
& Quasianalytic in all variables
& Compact manifold; no prescribed interior point correspondence{; full data, or inputs and observations on any nonempty open boundary portion} \\
\hline
Thm.~\ref{thm:normal-quasianalytic}
& Quasianalytic only in the normal variable; tangential dependence $C^\infty$
& Prescribed global boundary normal form{; full boundary data} \\
\hline
Thm.~\ref{thm:tangentially-homogeneous}
& $C^\infty$
& Prescribed product structure and tangential translation invariance; full boundary data \\
\hline
Thm.~\ref{thm:loewner} and Cor.~\ref{cor:loewner-metric}
& $C^\infty$
& Fixed Euclidean coordinates, a prescribed foliation, and layerwise one-sided Loewner ordering; partial boundary data \\
\hline
\end{tabular}
\caption{Comparison by regularity, data, and geometric structure.}
\label{tab:main-results}
\end{table}

In Section~\ref{sec:proof-global}, we combine boundary determination with quasianalytic continuation and Poisson embeddings to treat both full and partial data on general manifolds.
Section~\ref{sec:proof-normal} develops the continuation argument in the prescribed normal variable.
For the smooth results, Section~\ref{sec:proof-homogeneous} uses Fourier reduction and inverse spectral theory, while Section~\ref{sec:proof-loewner} combines monotonicity, localized potentials, and continuous layer stripping.
The latter section concludes with the metric interpretation.

\section{Quasianalytic uniqueness on general manifolds}\label{sec:proof-global}

This section is devoted to the proofs of Theorem~\ref{thm:global-quasianalytic} and Corollary~\ref{cor:global-quasianalytic-local}.
Subsection~\ref{subsec:global-regularity-boundary} establishes the quasianalytic regularity properties and boundary determination used in both proofs.
Subsection~\ref{subsec:global-full-data} proves Theorem~\ref{thm:global-quasianalytic} by globalizing a collar isometry with Poisson embeddings and then recovering the metric in harmonic coordinates.
Subsection~\ref{subsec:global-partial-data} proves Corollary~\ref{cor:global-quasianalytic-local} using boundary values supported in $\Gamma$ and the metric completion of the interior.

\subsection{Quasianalytic regularity and boundary determination}\label{subsec:global-regularity-boundary}

Fix an admissible quasianalytic weight sequence $\mathbf{W}$ as in Definition~\ref{def:admissible-quasianalytic-weight}, and let $\mathcal{E}^{\{\mathbf{W}\}}$ be the associated quasianalytic class defined in Definition~\ref{def:roumieu-denjoy-carleman-class}.
Throughout this subsection, $n\ge3$, $M$ is a compact connected $\mathcal{E}^{\{\mathbf{W}\}}$ manifold with nonempty boundary, and $g_1,g_2$ are $\mathcal{E}^{\{\mathbf{W}\}}$ Riemannian metrics on $M$.
For $f\in C^\infty(\partial M)$, let
$u_j^f:=u_{g_j}^f,
  \qquad j=1,2.$
Thus $u_j^f$ is the harmonic extension defined by~\eqref{eq:2} with $g=g_j$.
We write $\mathcal D'(U)$ for the space of distributions on an open set $U$.

We first record the regularity properties of $\mathcal{E}^{\{\mathbf{W}\}}$ needed for harmonic extensions and boundary normal coordinates.

\begin{lemma}\label{lem:dc-stability}
Let $\mathbf{W}$ be as in Definition~\ref{def:admissible-quasianalytic-weight}, and let $\mathcal{E}^{\{\mathbf{W}\}}$ be the associated quasianalytic class in Definition~\ref{def:roumieu-denjoy-carleman-class}.
The following assertions hold.
\begin{enumerate}[label=(\roman*)]
  \item The class $\mathcal{E}^{\{\mathbf{W}\}}$ is closed under differentiation, products, division by a nowhere zero function, composition, and local inversion of a map with nonsingular differential.
  The local solution of an ordinary differential equation with an $\mathcal{E}^{\{\mathbf{W}\}}$ right-hand side is of class $\mathcal{E}^{\{\mathbf{W}\}}$ in time, initial data, and parameters.
  \item Let $U$ be an open coordinate domain, let $P$ be an elliptic differential operator with coefficients in $\mathcal{E}^{\{\mathbf{W}\}}(U)$, and let $v\in\mathcal D'(U)$.
  Then
  \begin{equation}
    Pv\in\mathcal{E}^{\{\mathbf{W}\}}(U)\quad\Longrightarrow\quad v\in\mathcal{E}^{\{\mathbf{W}\}}(U).
    \label{eq:22}
  \end{equation}
  In particular, for every $f\in C^\infty(\partial M)$ and $j=1,2$,
  \begin{equation}
    u_j^f\in\mathcal{E}^{\{\mathbf{W}\}}(M^\circ).
    \label{eq:23}
  \end{equation}
  \item Let $U$ be a connected coordinate domain and let $a\in\mathcal{E}^{\{\mathbf{W}\}}(U)$.
  If $\partial^\alpha a(x_0)=0$ for every multiindex $\alpha$ at some $x_0\in U$, then $a=0$ on $U$.
  Consequently, for $a,b\in\mathcal{E}^{\{\mathbf{W}\}}(U)$,
  \begin{equation}
    a=b\text{ on a nonempty open subset of }U
    \quad\Longrightarrow\quad a=b\text{ on }U.
    \label{eq:24}
  \end{equation}
  \item Let $q\in\partial M$ and $j\in\{1,2\}$, and let $\widetilde M_j$ be a local $\mathcal{E}^{\{\mathbf{W}\}}$ extension of $M$ across the boundary near $q$ on which $g_j$ has an $\mathcal{E}^{\{\mathbf{W}\}}$ extension $\widetilde g_j$.
  There are a neighborhood $V\subset\partial M$ of $q$ and $\varepsilon>0$ such that the boundary normal collar map
  \[
    \Phi_j:V\times(-\varepsilon,\varepsilon)\longrightarrow\widetilde M_j,
    \qquad
    \Phi_j(y,t)=\exp_y^{\widetilde g_j}\bigl(-t\nu_{g_j}(y)\bigr),
  \]
  is an $\mathcal{E}^{\{\mathbf{W}\}}$ diffeomorphism onto its image.
\end{enumerate}
\end{lemma}

\begin{proof}
Conditions \eqref{eq:A1}--\eqref{eq:A2} say precisely that $\mathbf{W}$ is a regular weight sequence in the convention normalized by $W_k/k!$, while \eqref{eq:A3} is the moderate growth condition required for elliptic regularity.
Differentiation and reciprocal closure, composition and local inversion, and dependence of ODE solutions on time, initial data, and parameters follow from the regular weight sequence results in~\cite[Definition~2.4, Remark~2.5, and Theorems~2.6--2.7]{furdos2020geometric}.
Closure under products follows by composing with the real-analytic multiplication map, and division follows from products and reciprocal closure.
This proves (i).

The projection characterization of the Denjoy--Carleman wavefront set and microlocal elliptic regularity give \eqref{eq:22}; see~\cite[Theorems~3.5 and~6.1]{furdos2020geometric}.
After expansion in local coordinates, every coefficient of $\Delta_{g_j}$ belongs to $\mathcal{E}^{\{\mathbf{W}\}}$.
Indeed, differentiation and products preserve this class, the adjugate formula gives $g_j^{-1}$, and composition with the positive real-analytic functions $s\mapsto s^{1/2}$ and $s\mapsto s^{-1/2}$ gives $|g_j|^{\pm1/2}$.
Since $\Delta_{g_j}u_j^f=0$ in $M^\circ$, equation~\eqref{eq:22} yields \eqref{eq:23}, even when $f$ does not belong to the quasianalytic class.
This proves (ii).

Condition \eqref{eq:A4} is the Denjoy--Carleman quasianalyticity criterion, which makes the Taylor map at every point of a connected coordinate domain injective~\cite[Theorem~2]{thilliez2008quasianalytic}.
Applying this result to $a-b$ gives \eqref{eq:24}.
This proves (iii).

After shrinking the chosen extension near $q$, the tensor $\widetilde g_j$ remains positive definite.
The algebraic closure properties in (i) show that the outward unit normal field $\nu_{g_j}$ is of class $\mathcal{E}^{\{\mathbf{W}\}}$ on $V$.
The Christoffel symbols of $\widetilde g_j$ are also of this class, so the ODE conclusion in (i) shows that $\Phi_j$ is of class $\mathcal{E}^{\{\mathbf{W}\}}$ in $(y,t)$.
At $t=0$ its differential is
\[
  D\Phi_j|_{(y,0)}(v,s)=v-s\nu_{g_j}(y),
  \qquad v\in T_y\partial M,
\]
and is an isomorphism from $T_y\partial M\times\mathbb{R}$ to $T_y\widetilde M_j$.
The local inversion conclusion in (i) therefore shows that the inverse boundary normal coordinate map is of class $\mathcal{E}^{\{\mathbf{W}\}}$.
This proves (iv).
\end{proof}

The following lemma states the localized smooth boundary determination result needed below for the scalar Laplace--Beltrami operator.

\begin{lemma}\label{lem:local-boundary-jets}
Let $n\ge3$, let $g_1,g_2$ be smooth Riemannian metrics on $M$, and let $q\in\partial M$.
Let $\chi\in C_c^\infty(\partial M)$ equal one near $q$, and suppose
\[
  \chi\Lambda_{g_1}\chi=\chi\Lambda_{g_2}\chi.
\]
Use the same tangential coordinates near $q$ and boundary normal coordinates for the two metrics, so that
\begin{equation}
  g_j=\,\mathrm{d} t^2+h_j(t,y).
  \label{eq:local-boundary-normal-form}
\end{equation}
Then, after possibly shrinking the boundary chart,
\begin{equation}
  \partial_t^kh_1(0,y)=\partial_t^kh_2(0,y)
  \qquad(k=0,1,2,\ldots).
  \label{eq:local-boundary-jets}
\end{equation}
\end{lemma}

\begin{proof}
The scalar boundary determination theorem states that, for a smooth Riemannian metric on an $n$-dimensional manifold with $n>2$, the full symbol of its DtN map on a boundary coordinate neighborhood determines the full Taylor series of the metric at the boundary in boundary normal coordinates~\cite[Proposition~1.1]{lee1989determining}; see also the $0$-form case of~\cite[Theorem~1.1]{joshi2005inverse}.
Since $\chi=1$ near $q$, the identity $\chi\Lambda_{g_1}\chi=\chi\Lambda_{g_2}\chi$ implies that the full symbols of $\Lambda_{g_1}$ and $\Lambda_{g_2}$ agree on a smaller boundary neighborhood of $q$.
Indeed, where $\chi=1$, localization preserves the Schwartz kernel near the diagonal and hence the full symbol modulo smoothing terms.
The cited theorem therefore gives \eqref{eq:local-boundary-jets} after shrinking the boundary chart.
\end{proof}

The representations in \eqref{eq:local-boundary-normal-form} are obtained from local boundary normal coordinates constructed separately for the two metrics and are used only to compare their boundary jets.

By contrast, Theorem~\ref{thm:normal-quasianalytic} assumes a prescribed product structure that provides common global boundary normal coordinates on all of $M$, fixes the interior point correspondence, and reduces quasianalytic continuation to the normal variable; see Remark~\ref{rem:normal-geometry}.

\subsection{Uniqueness from full boundary data}\label{subsec:global-full-data}
We now prove Theorem~\ref{thm:global-quasianalytic}.
Smooth boundary determination and Runge approximation are used as standard inputs.
The two quasianalytic continuation steps are contained in Propositions~\ref{prop:common-collar} and~\ref{prop:poisson-globalization}: the first promotes boundary jets to a collar isometry, and the second globalizes this isometry by the Poisson embedding approach.
Lemmas~\ref{lem:runge-approximation}--\ref{lem:poisson-geometry} provide the global harmonic functions and harmonic coordinates used in this continuation.
Lemma~\ref{lem:metric-from-harmonic} completes the recovery of the metric.

The next proposition is the first key quasianalytic step in the proof: it upgrades the formal boundary agreement supplied by Lemma~\ref{lem:local-boundary-jets} to an actual geometric identification on an open collar.
The essential bridge is the availability of two-sided boundary normal coordinates in the quasianalytic class $\mathcal{E}^{\{\mathbf{W}\}}$, which makes the Denjoy--Carleman identity principle applicable at the boundary.

\begin{proposition}\label{prop:common-collar}
There exists $\varepsilon>0$ such that, for $j=1,2$, the boundary normal map
\[
  \Phi_j:\partial M\times[0,\varepsilon)\longrightarrow M,
  \qquad
  \Phi_j(y,t)
  =\exp_y^{g_j}\bigl(-t\nu_{g_j}(y)\bigr),
\]
is a diffeomorphism in the quasianalytic class $\mathcal{E}^{\{\mathbf{W}\}}$ onto its image
\[
  C_j:=\Phi_j\bigl(\partial M\times[0,\varepsilon)\bigr).
\]
Here $\nu_{g_j}$ is the outward unit normal, so $C_j$ is the collar obtained by moving each $y\in\partial M$ inward by a $g_j$-normal distance $t<\varepsilon$.
The map that matches the same boundary point and the same normal distance,
\[
  F_0:=\Phi_2\circ\Phi_1^{-1}:C_1\longrightarrow C_2,
\]
satisfies
\[
  F_0|_{\partial M}=\operatorname{Id},
  \qquad
  F_0^*g_2=g_1
  \quad\text{on }C_1.
\]
Moreover,
\[
  u_1^f=u_2^f\circ F_0
  \quad\text{on }C_1
\]
for every $f\in C^\infty(\partial M)$.
\end{proposition}

\begin{proof}
Fix $p\in\partial M$ and choose a boundary coordinate neighborhood $V_p\subset\partial M$ of $p$.
For $j=1,2$, let $\widetilde g_j$ be an $\mathcal{E}^{\{\mathbf{W}\}}$ extension of $g_j$ across the boundary near $p$, defined on a local extension $\widetilde M_j$ of $M$.
After shrinking $V_p$, Lemma~\ref{lem:dc-stability}(iv) provides $\varepsilon_p>0$ and $\mathcal{E}^{\{\mathbf{W}\}}$ boundary normal coordinate maps
\[
  \Phi_{j,p}:V_p\times(-\varepsilon_p,\varepsilon_p)
  \longrightarrow \widetilde M_j,
  \qquad
  \Phi_{j,p}(y,t)
  =\exp_y^{\widetilde g_j}\bigl(-t\nu_{g_j}(y)\bigr).
\]
Here $t>0$ points into $M$, since $\nu_{g_j}$ is the outward unit normal.
The same tangential coordinate $y$ is used for the two metrics, whereas the normal geodesics are constructed separately.
In these coordinates,
$\Phi_{j,p}^*\widetilde g_j
  =\,\mathrm{d} t^2+h_{j,p}(t,y),$
where each component of $h_{j,p}$ belongs to $\mathcal{E}^{\{\mathbf{W}\}}$ on the two-sided coordinate neighborhood.

Choose $\chi\in C_c^\infty(V_p)$ with $\chi=1$ on a smaller neighborhood of $p$.
The full-data hypothesis \eqref{eq:3} implies
$\chi\Lambda_{g_1}\chi=\chi\Lambda_{g_2}\chi.$
Lemma~\ref{lem:local-boundary-jets} therefore gives, after a further shrinking of $V_p$,
\begin{equation}
  \partial_t^kh_{1,p}(0,y)=\partial_t^kh_{2,p}(0,y)
  \qquad(k=0,1,2,\ldots),
  \label{eq:25}
\end{equation}
for every $y\in V_p$.
Since \eqref{eq:25} is an identity of smooth tensor-valued functions of $y$, it may be differentiated tangentially.
Thus, for every tangential multiindex $\alpha$, every $k\ge0$, and every coordinate component,
\[
  \partial_y^\alpha\partial_t^k
  \bigl(h_{1,p}-h_{2,p}\bigr)(0,y)=0,
  \qquad y\in V_p.
\]
Let $y_0\in V_p$ represent $p$.
The preceding identities say that the full Taylor series at $(y_0,0)$ of each component of $h_{1,p}-h_{2,p}$ is zero.
Since the two-sided extensions make $(y_0,0)$ an interior point, Lemma~\ref{lem:dc-stability}(iii) shows that every component of $h_{1,p}-h_{2,p}$ vanishes on a connected neighborhood of $(y_0,0)$.
After shrinking $V_p$ and $\varepsilon_p$ simultaneously for all tensor components, we obtain
\[
  h_{1,p}(t,y)=h_{2,p}(t,y)
  \qquad\text{for }(y,t)\in V_p\times(-\varepsilon_p,\varepsilon_p).
\]

On the side $t\ge0$, define the local collars
\[
  C_{j,p}
  :=\Phi_{j,p}\bigl(V_p\times[0,\varepsilon_p)\bigr),
  \qquad j=1,2,
\]
and the map
\[
  F_p:C_{1,p}\longrightarrow C_{2,p},
  \qquad
  F_p=\Phi_{2,p}\circ\Phi_{1,p}^{-1}.
\]
The equality of the two coordinate expressions for the metrics gives
$F_p^*g_2=g_1
  \quad\text{on }C_{1,p}.$
Moreover, $F_p(y)=y$ for $y\in V_p$, and Lemma~\ref{lem:dc-stability}(i) shows that $F_p$ and $F_p^{-1}$ are of class $\mathcal{E}^{\{\mathbf{W}\}}$.

We next verify that the local maps $F_p$ are compatible.
A point in a sufficiently thin $g_1$-normal collar has a unique representation
\[
  x=\exp_y^{g_1}\bigl(-t\nu_{g_1}(y)\bigr),
  \qquad y\in\partial M,
  \quad t\ge0.
\]
Every local map whose domain contains $x$ sends it to
$\exp_y^{g_2}\bigl(-t\nu_{g_2}(y)\bigr).$
Hence two local maps agree wherever their sufficiently small collar domains overlap.
Because $\partial M$ is compact, finitely many of the neighborhoods $V_p$ cover it, and the collar widths may be reduced to one common number $\varepsilon>0$.
We choose this width within the normal collar radii of both metrics; it is fixed independently of $f$.
The boundary normal maps then define diffeomorphisms
\[
  \Phi_j:\partial M\times[0,\varepsilon)\longrightarrow C_j,
  \qquad
  \Phi_j(y,t)=\exp_y^{g_j}\bigl(-t\nu_{g_j}(y)\bigr),
\]
onto collars $C_j\subset M$.
The compatible local identifications combine to give
\begin{equation}
  F_0:=\Phi_2\circ\Phi_1^{-1}:C_1\longrightarrow C_2,
  \qquad F_0^*g_2=g_1.
  \label{eq:26}
\end{equation}
The map $F_0$ fixes $\partial M$ pointwise and is an $\mathcal{E}^{\{\mathbf{W}\}}$ diffeomorphism by the local construction.
This completes the geometric part of the argument.

It remains to compare the harmonic extensions.
Fix $f\in C^\infty(\partial M)$.
For $j=1,2$, recall from \eqref{eq:2} that $u_j^f=u_{g_j}^f$ is the $g_j$-harmonic extension of $f$ to $M$.
On $C_1$, pull back the $g_2$-harmonic extension and compare it with the $g_1$-harmonic extension by setting
\[
  v^f:=F_0^*u_2^f=u_2^f\circ F_0
  \qquad\text{and}\qquad
  w^f:=u_1^f-v^f.
\]
Since $F_0$ is an isometry, invariance of the Laplace--Beltrami operator under pullback gives
$\Delta_{g_1}v^f
  =\bigl(\Delta_{g_2}u_2^f\bigr)\circ F_0
  =0.$
Both $u_1^f$ and $v^f$ have boundary trace $f$, because $F_0$ fixes the boundary.
The definition of the boundary normal maps also gives
$D F_0(y)\nu_{g_1}(y)=\nu_{g_2}(y),
  \qquad y\in\partial M.$
Consequently, the normal derivative of $v^f$ satisfies
\[
  \left.\partial_{\nu_{g_1}}v^f\right|_{\partial M}
  =\left.\partial_{\nu_{g_2}}u_2^f\right|_{\partial M}
  =\Lambda_{g_2}f
  =\Lambda_{g_1}f
  =\left.\partial_{\nu_{g_1}}u_1^f\right|_{\partial M},
\]
where the middle equality is precisely \eqref{eq:3}.
Thus $w^f$ satisfies
\[
  \Delta_{g_1}w^f=0\quad\text{in }C_1,
  \qquad
  w^f|_{\partial M}=0,
  \qquad
  \partial_{\nu_{g_1}}w^f|_{\partial M}=0.
\]

The function $w^f$ has zero Cauchy data on $\partial M$.
Using the prescribed extension of $g_1$, extend $w^f$ by zero across the boundary.
The zero Dirichlet and normal traces make this extension a weak solution of a uniformly elliptic divergence-form equation with Lipschitz principal coefficients.
Indeed, the zero Dirichlet trace places the extension in $H^1_{\mathrm{loc}}$, while the zero normal trace eliminates the boundary term in the weak formulation.
Since it vanishes on an exterior open set, the standard unique continuation property gives $w^f=0$ near $\partial M$~\cite{isakov2006inverse}.
On each connected component of the fixed collar $C_1$, interior unique continuation propagates this vanishing to the whole component.
Since $C_1$ was chosen before $f$, the same collar works for every boundary value, and hence
\begin{equation}
  u_1^f=u_2^f\circ F_0
  \quad\text{on }C_1,
  \qquad f\in C^\infty(\partial M).
  \label{eq:27}
\end{equation}

The proof is complete.
\end{proof}

The qualitative Runge approximation allows global $g_j$-harmonic solutions generated by Dirichlet data supported in any nonempty open part of the boundary to approximate a prescribed local $g_j$-harmonic function on an interior set; it is a standard consequence of duality and unique continuation~\cite{browder1962approximation,garofalo1986monotonicity,lassas2020poisson,ruland2019quantitative}.

\begin{lemma}\label{lem:runge-approximation}
Fix $j\in\{1,2\}$, and let $\Gamma\subset\partial M$ be nonempty and relatively open.
Let $U\Subset M^\circ$ be an open set with finitely many connected components and smooth boundary such that $M\setminus\overline U$ is connected.
Let $v$ be a local $g_j$-harmonic function, defined in an open neighborhood of $\overline U$ and satisfying $\Delta_{g_j}v=0$ there.
Then, for every $U'\Subset U$ and every $r\in\mathbb{N}_0$, there exists a sequence $(f_\ell)_{\ell=1}^\infty\subset C_c^\infty(\Gamma)$ such that the corresponding solutions $u_j^{f_\ell}=u_{g_j}^{f_\ell}$, with the boundary values extended by zero to $\partial M$, satisfy
\begin{equation}
  u_j^{f_\ell}\longrightarrow v
  \quad\text{in }C^r(U')
  \quad\text{as }\ell\to\infty.
  \label{eq:28}
\end{equation}
\end{lemma}

For a nonempty relatively open set $\Gamma\subset\partial M$, define the family of global harmonic extensions of boundary values supported in $\Gamma$ by
\begin{equation}
  \mathcal H_j^\Gamma
  :=\{u_j^f:f\in C_c^\infty(\Gamma)\},
  \qquad
  \mathcal H(g_j)
  :=\mathcal H_j^{\partial M}
  =\{u\in C^\infty(M):\Delta_{g_j}u=0\}
  =\{u_{g_j}^f:f\in C^\infty(\partial M)\}.
  \label{eq:harmonic-family}
\end{equation}
Lemma~\ref{lem:runge-approximation} gives the following consequences uniformly in $\Gamma$.

\begin{lemma}\label{lem:runge-jets}
Fix $j\in\{1,2\}$, and let $\Gamma\subset\partial M$ be nonempty and relatively open.
The family $\mathcal H_j^\Gamma$ defined in \eqref{eq:harmonic-family} has the following properties.
\begin{description}
  \item[(R1)] If $x_1\ne x_2$ are points of $M^\circ$, there is $u\in\mathcal H_j^\Gamma$ such that $u(x_1)\ne u(x_2)$.
  \item[(R2)] If $x\in M^\circ$ and $\xi\in T_x^*M$, there is $u\in\mathcal H_j^\Gamma$ such that $\,\mathrm{d}u(x)=\xi$.
\end{description}
Moreover, $\mathcal H_j^\Gamma\subset\mathcal E^{\{\mathbf W\}}(M^\circ)$, and every $x\in M^\circ$ has a local $g_j$-harmonic coordinate map in this class whose components belong to $\mathcal H_j^\Gamma$.
When $\Gamma=\partial M$, the full family $\mathcal H(g_j)$ also separates arbitrary distinct points of $M$, including boundary points.
\end{lemma}

\begin{proof}
For \textup{(R1)}, choose disjoint coordinate balls $B_1,B_2\Subset M^\circ$ centered at $x_1,x_2$ such that $M\setminus(\overline{B_1}\cup\overline{B_2})$ is connected.
The function that equals zero on $B_1$ and one on $B_2$ is harmonic on $B_1\cup B_2$.
Lemma~\ref{lem:runge-approximation} with boundary set $\Gamma$ and $r=0$ gives a member of $\mathcal H_j^\Gamma$ whose values at $x_1$ and $x_2$ are different.

For \textup{(R2)}, choose local $g_j$-harmonic coordinates near $x$~\cite{deturck1981some}.
A linear combination of their coordinate functions has any prescribed differential $\xi\in T_x^*M$.
Lemma~\ref{lem:runge-approximation} with boundary set $\Gamma$ and $r=1$ shows that $\xi$ lies in the closure of
$\{\,\mathrm{d}u(x):u\in\mathcal H_j^\Gamma\}$
in $T_x^*M$.
This set is a linear subspace of a finite-dimensional vector space and hence is closed, which proves \textup{(R2)}.

Equation~\eqref{eq:23} gives the stated quasianalytic regularity.
Applying \textup{(R2)} to a basis of $T_x^*M$ and using the inverse function theorem gives the coordinate assertion.

It remains to prove the assertion that harmonic functions separate all points of $M$ when $\Gamma=\partial M$.
The interior case is \textup{(R1)}.
If both points lie on the boundary, choose $f\in C^\infty(\partial M)$ taking different values at them.
If $x_1\in M^\circ$ and $x_2\in\partial M$, choose $f\ge0$, $f\not\equiv0$, with $f(x_2)=0$; the strong maximum principle gives $u_{g_j}^f(x_1)>0=u_{g_j}^f(x_2)$.
The remaining mixed case follows by interchanging the points.

The proof is complete.
\end{proof}

We use the Poisson embedding introduced in~\cite{lassas2020poisson} to encode the global harmonic family.
Its standard mapping properties are recalled in Lemma~\ref{lem:poisson-geometry}, where we also establish the quasianalytic regularity of the induced local point correspondences.

\begin{definition}\label{def:poisson-embedding}
For each $j\in\{1,2\}$, define the \emph{Poisson embedding} associated with $g_j$ by
\begin{equation}
  P_j:M\longrightarrow\mathcal D'(\partial M),
  \qquad P_j(x)(f):=u_j^f(x).
  \label{eq:30}
\end{equation}
Here $x\in M$, $f\in C^\infty(\partial M)$, and $u_j^f=u_{g_j}^f$ is the solution of the Dirichlet problem~\eqref{eq:2} with boundary value $f$.
For each fixed $x$, elliptic boundary regularity implies that $f\mapsto u_j^f(x)$ is a continuous linear functional on $C^\infty(\partial M)$, so $P_j(x)$ is an element of $\mathcal D'(\partial M)$.
If $x\in\partial M$, then $P_j(x)=\delta_x$, where $\delta_x$ denotes the Dirac distribution at $x$.
\end{definition}

Thus $P_j(x)$ records the values at $x$ of all global $g_j$-harmonic functions.

\begin{lemma}\label{lem:poisson-geometry}
Let $P_j$, $j\in\{1,2\}$, be the Poisson embeddings in Definition~\ref{def:poisson-embedding}.
Each $P_j$ is injective, its restriction to $M^\circ$ is an immersion, and $P_j:M\to P_j(M)$ is a homeomorphism.
For every $x\in M^\circ$, there exist $f_1,\ldots,f_n\in C^\infty(\partial M)$ such that
$(u_j^{f_1},\ldots,u_j^{f_n})$ is a local coordinate map near $x$ in the quasianalytic class $\mathcal{E}^{\{\mathbf{W}\}}$.
Consequently, for $i,j\in\{1,2\}$, any locally defined map $P_j^{-1}\circ P_i$ between open subsets of $M^\circ$ belongs to this class.
On the collars $C_1,C_2$ of Proposition~\ref{prop:common-collar}, one has $P_1=P_2\circ F_0$ on $C_1$, and hence $P_2^{-1}\circ P_1=F_0$ there.
\end{lemma}

\begin{proof}
Fix $j\in\{1,2\}$.
Applied with $\Gamma=\partial M$, the standard mapping properties of the Poisson embedding show that $P_j$ is continuous on $M$, smooth on $M^\circ$, and has injective differential there~\cite[Propositions~2.1 and~2.3]{lassas2020poisson}.
Thus $P_j|_{M^\circ}$ is an immersion.

If $P_j(x_1)=P_j(x_2)$, then Definition~\ref{def:poisson-embedding} gives $u_j^f(x_1)=u_j^f(x_2)$ for every $f\in C^\infty(\partial M)$.
The separation of points in Lemma~\ref{lem:runge-jets} with $\Gamma=\partial M$ therefore gives $x_1=x_2$, so $P_j$ is injective.
Since $M$ is compact and $\mathcal D'(\partial M)$ is Hausdorff, the continuity of $P_j$ makes it a homeomorphism onto its image.

We next prove the quasianalytic assertions.
Fix $x\in M^\circ$.
By Lemma~\ref{lem:runge-jets} with $\Gamma=\partial M$, there are globally defined functions $u^1,\ldots,u^n\in\mathcal H(g_j)$ that form a local coordinate map of class $\mathcal E^{\{\mathbf W\}}$ near $x$.
By the definition of $\mathcal H(g_j)$ in~\eqref{eq:harmonic-family}, write $u^k=u_j^{f_k}$ for boundary values $f_k\in C^\infty(\partial M)$.
Set
\begin{equation}
  U_j=(u_j^{f_1},\ldots,u_j^{f_n}).
  \label{eq:31}
\end{equation}
After restricting the coordinate neighborhood, there are open sets $V_j\subset M^\circ$ and $\Omega_j\subset\mathbb R^n$ such that $U_j:V_j\to\Omega_j$ is a coordinate diffeomorphism.
Lemma~\ref{lem:dc-stability}\textup{(i)} shows that the local inverse $U_j^{-1}:\Omega_j\to V_j$ also belongs to $\mathcal E^{\{\mathbf W\}}$.
This proves the coordinate assertion in the statement.

We next make the regularity of a local point correspondence explicit.
Fix $i,j\in\{1,2\}$, and suppose that $F_{ij}:=P_j^{-1}\circ P_i$ is defined on an open neighborhood $V_i\subset M^\circ$ of a point $p$ and takes values in $M^\circ$.
Set $q=F_{ij}(p)$, and apply the preceding construction at $q$ to obtain boundary values $f_1,\ldots,f_n$ and a coordinate diffeomorphism $U_j:V_j\to\Omega_j$.
Define the continuous linear map
\[
  \operatorname{ev}_{f_1,\ldots,f_n}:\mathcal D'(\partial M)\longrightarrow\mathbb R^n,
  \qquad
  \operatorname{ev}_{f_1,\ldots,f_n}(T)
  =\bigl(T(f_1),\ldots,T(f_n)\bigr).
\]
For every $y\in V_j$, Definition~\ref{def:poisson-embedding} gives
\[
  \operatorname{ev}_{f_1,\ldots,f_n}(P_j(y))
  =\bigl(u_j^{f_1}(y),\ldots,u_j^{f_n}(y)\bigr)
  =U_j(y).
\]
Consequently, on $P_j(V_j)$ one has
\begin{equation}
  P_j^{-1}
  =U_j^{-1}\circ\operatorname{ev}_{f_1,\ldots,f_n}.
  \label{eq:32}
\end{equation}
The map $F_{ij}$ is continuous because $P_i$ is continuous and $P_j^{-1}:P_j(M)\to M$ is continuous by the homeomorphism proved above.
After decreasing $V_i$, we may therefore assume that $F_{ij}(V_i)\subset V_j$.
For $y\in V_i$, formula~\eqref{eq:32} then gives
$F_{ij}(y)
  =U_j^{-1}\bigl(u_i^{f_1}(y),\ldots,u_i^{f_n}(y)\bigr).$
By equation~\eqref{eq:23}, every component $u_i^{f_k}$ of the inner map belongs to $\mathcal E^{\{\mathbf W\}}(V_i)$, while $U_j^{-1}$ belongs to the same class by Lemma~\ref{lem:dc-stability}\textup{(i)}.
The closure of this class under composition, again from Lemma~\ref{lem:dc-stability}\textup{(i)}, proves that $F_{ij}=P_j^{-1}\circ P_i$ is of class $\mathcal E^{\{\mathbf W\}}$ on $V_i$.

Finally, let $C_1,C_2$ and $F_0:C_1\to C_2$ be the collars and the collar isometry in Proposition~\ref{prop:common-collar}.
For every $x\in C_1$ and $f\in C^\infty(\partial M)$, equation~\eqref{eq:27} gives
$P_1(x)(f)
  =u_1^f(x)
  =u_2^f(F_0(x))
  =P_2(F_0(x))(f).$
Since this equality holds for every $f$, it is an equality in $\mathcal D'(\partial M)$, and hence
\begin{equation}
  P_1=P_2\circ F_0
  \label{eq:33}
\end{equation}
on $C_1$.
The injectivity of $P_2$ now yields $P_2^{-1}\circ P_1=F_0$ on $C_1$.

The proof is complete.
\end{proof}

The following proposition extends the collar map $F_0$ obtained in Proposition~\ref{prop:common-collar} to a global diffeomorphism in the quasianalytic class.

\begin{proposition}\label{prop:poisson-globalization}
There exists a unique diffeomorphism $F:M\to M$ in the quasianalytic class $\mathcal{E}^{\{\mathbf{W}\}}$ such that
\begin{equation*}
  F|_{C_1}=F_0,
  \qquad
  u_1^f=u_2^f\circ F
  \quad\text{on }M
\end{equation*}
for every $f\in C^\infty(\partial M)$.
In particular, $F|_{\partial M}=\operatorname{Id}$.
\end{proposition}

\begin{proof}
Let $B$ be the union of all open subsets $O\subset M$ such that
\begin{equation*}
  C_1\subset O
  \qquad\text{and}\qquad
  P_1(O)\subset P_2(M).
\end{equation*}
Equation~\eqref{eq:33} shows that $C_1$ is one such open set, and hence $C_1\subset B$.
Since $B$ is a union of open sets, it is open, and its definition gives
\begin{equation}
  P_1(B)\subset P_2(M).
  \label{eq:34}
\end{equation}
Every open subset $O\subset M$ satisfying $C_1\subset O$ and $P_1(O)\subset P_2(M)$ occurs in the union defining $B$, and hence $O\subset B$.
Thus $B$ is the largest open subset of $M$ satisfying these two inclusions.
On $B$ define
\begin{equation}
  F=P_2^{-1}\circ P_1.
  \label{eq:35}
\end{equation}
Equations~\eqref{eq:33} and~\eqref{eq:35}, together with the injectivity of $P_2$ in Lemma~\ref{lem:poisson-geometry}, give $F=F_0$ on $C_1$.

First, $F(B\cap M^\circ)\subset M^\circ$.
Otherwise, if $x\in B\cap M^\circ$ and $F(x)=y_\partial\in\partial M$, equation~\eqref{eq:35} and Definition~\ref{def:poisson-embedding} would give $u_1^f(x)=f(y_\partial)$ for every $f$.
Choosing $f\ge0$, $f\not\equiv0$, and $f(y_\partial)=0$ contradicts the strong maximum principle.
Lemma~\ref{lem:poisson-geometry} therefore shows that $F$ is of class $\mathcal{E}^{\{\mathbf{W}\}}$ near every point of $B\cap M^\circ$.
On $C_1$, equations~\eqref{eq:33} and~\eqref{eq:35} give $F=F_0$, and Proposition~\ref{prop:common-collar} states that $F_0$ is of class $\mathcal{E}^{\{\mathbf{W}\}}$ up to $\partial M$.
Moreover,
\begin{equation}
  u_1^f=u_2^f\circ F\quad\text{on }B,
  \qquad f\in C^\infty(\partial M).
  \label{eq:36}
\end{equation}
If $DF(x)X=0$ at an interior point, differentiating \eqref{eq:36} gives $\,\mathrm{d} u_1^f(x)X=0$ for every $f$, and Lemma~\ref{lem:runge-jets}\textup{(R2)} with $\Gamma=\partial M$ for $g_1$ implies $X=0$.
Since the dimensions agree, $DF(x)$ is invertible.
Thus $F$ is a local diffeomorphism at every point of $B\cap M^\circ$.
At each point of $\partial M$, equations~\eqref{eq:33} and~\eqref{eq:35} give $F=F_0$ on $C_1$, and Proposition~\ref{prop:common-collar} states that $F_0:C_1\to C_2$ is a diffeomorphism.
Hence $F$ is a local diffeomorphism on all of $B$.

It remains to prove that $B$ is closed.
Suppose otherwise, and choose $p_k\in B$ with $p_k\to x_1\in\partial B\setminus B$.
We shall show that the candidate map $F=P_2^{-1}\circ P_1$ extends to an open neighborhood of $x_1$, which would contradict the maximality of $B$.
Since $C_1\subset B$ by the definition of $B$ and $C_1$ is a collar of $\partial M$ by Proposition~\ref{prop:common-collar}, we have $x_1\in M^\circ$.
Passing to a subsequence, compactness gives $F(p_k)\to x_2\in M$.
Thus $x_1$ is the point across which $F$ is to be extended, while $x_2$ is its candidate image along the chosen subsequence.
We first show that this candidate image lies in $M^\circ$.
If $x_2\in\partial M$, then \eqref{eq:36}, continuity, and the boundary condition $u_2^f|_{\partial M}=f$ in \eqref{eq:2} give
\begin{equation}
  u_1^f(x_1)=f(x_2)\quad\text{for every }f.
  \label{eq:37}
\end{equation}
Choose $f\ge0$, $f\not\equiv0$, and $f(x_2)=0$.
The strong maximum principle gives $u_1^f(x_1)>0$, a contradiction.

We have proved that both $x_1$ and $x_2$ lie in $M^\circ$.
We now construct a candidate extension of $F$ across $x_1$ by matching harmonic functions near these two points.
Use Lemma~\ref{lem:runge-jets}\textup{(R2)} with $\Gamma=\partial M$ for $g_2$ to choose $f_1,\ldots,f_n$ such that
\begin{equation}
  U_2=(u_2^{f_1},\ldots,u_2^{f_n})
  \label{eq:38}
\end{equation}
is a harmonic coordinate map near $x_2$, and set
\begin{equation}
  U_1=(u_1^{f_1},\ldots,u_1^{f_n}).
  \label{eq:39}
\end{equation}
Taking the limit in \eqref{eq:36} along the sequence $(p_k)$ gives $U_1(x_1)=U_2(x_2)$.
Choose an open neighborhood $\Omega_2\subset M^\circ$ of $x_2$ such that
\begin{equation*}
  U_2:\Omega_2\longrightarrow U_2(\Omega_2)\subset\mathbb{R}^n
\end{equation*}
is a diffeomorphism onto the open set $U_2(\Omega_2)$.
Since $U_1(x_1)=U_2(x_2)\in U_2(\Omega_2)$ and $U_1$ is continuous, there exists an open neighborhood $\Omega_1\subset M^\circ$ of $x_1$ such that $U_1(\Omega_1)\subset U_2(\Omega_2)$.
We may therefore define the candidate extension
\begin{equation*}
  \widetilde F:\Omega_1\longrightarrow\Omega_2,
  \qquad
  \widetilde F=U_2^{-1}\circ U_1.
\end{equation*}
At this stage, $\Omega_1$ is only a neighborhood on which this composition is defined.
The map $\widetilde F$ is smooth there, but it remains to prove that it is a local diffeomorphism and that it preserves all harmonic extensions.

We first prove that $D\widetilde F(x_1)$ is invertible.
For all sufficiently large $k$, we have $p_k\in\Omega_1$ and $F(p_k)\in\Omega_2$.
Since $F$ is continuous, there is an open neighborhood $V_k\subset B\cap\Omega_1$ of $p_k$ such that $F(V_k)\subset\Omega_2$.
Equations~\eqref{eq:36}, \eqref{eq:38}, and~\eqref{eq:39} give $U_2\circ F=U_1$ on $V_k$, and hence $F=U_2^{-1}\circ U_1=\widetilde F$ there.
Assume that $D\widetilde F(x_1)X=0$, and choose $X_k\in T_{p_k}M$ converging to $X$ in a local trivialization of $TM$.
Differentiating \eqref{eq:36} gives
\begin{equation}
  \,\mathrm{d} u_1^f(p_k)X_k
  =\,\mathrm{d} u_2^f(F(p_k))D\widetilde F(p_k)X_k.
  \label{eq:40}
\end{equation}
The right-hand side of \eqref{eq:40} converges to $\,\mathrm{d}u_2^f(x_2)D\widetilde F(x_1)X=0$, so $\,\mathrm{d}u_1^f(x_1)X=0$ for every $f$.
Lemma~\ref{lem:runge-jets}\textup{(R2)} with $\Gamma=\partial M$ for $g_1$ provides $f$ such that $\,\mathrm{d}u_1^f(x_1)=g_1(X,\cdot)$.
Consequently, $0=\,\mathrm{d}u_1^f(x_1)X=g_1(X,X)$, and hence $X=0$.
Thus $D\widetilde F(x_1)$ is invertible.
By the inverse function theorem, we may now shrink $\Omega_1$ around $x_1$ so that
\begin{equation*}
  U_1:\Omega_1\longrightarrow U_1(\Omega_1)\subset\mathbb{R}^n
\end{equation*}
is a coordinate diffeomorphism and $U_1(\Omega_1)$ is connected.
The construction of $\widetilde F$ and the verification that $D\widetilde F(x_1)$ is invertible have so far used only smooth arguments.
It remains to show that $\widetilde F$ preserves every harmonic extension, rather than only the coordinate functions in \eqref{eq:38}--\eqref{eq:39}.
This is the point at which the quasianalytic assumptions enter.
After the final shrinking, choose $k$ sufficiently large and a nonempty open neighborhood $V\subset B\cap\Omega_1$ of $p_k$ such that $F(V)\subset\Omega_2$.
Equations~\eqref{eq:36}, \eqref{eq:38}, and~\eqref{eq:39} give $U_2\circ F=U_1$ on $V$, so $F=\widetilde F$ there.
For any $f$, the two functions
\begin{equation}
  u_1^f\circ U_1^{-1},
  \qquad \left.(u_2^f\circ U_2^{-1})\right|_{U_1(\Omega_1)}
  \label{eq:41}
\end{equation}
belong to $\mathcal{E}^{\{\mathbf{W}\}}(U_1(\Omega_1))$ by \eqref{eq:23} and Lemma~\ref{lem:dc-stability}\textup{(i)}.
They agree on the nonempty open set $U_1(V)$ by \eqref{eq:36} and the equality $F=\widetilde F$ on $V$.
The quasianalytic identity principle in Lemma~\ref{lem:dc-stability}\textup{(iii)} therefore makes them equal throughout the connected set $U_1(\Omega_1)$.
Since $f$ was arbitrary, Definition~\ref{def:poisson-embedding} gives
\begin{equation}
  P_1(x)=P_2(U_2^{-1}(U_1(x)))
  \qquad(x\in\Omega_1).
  \label{eq:42}
\end{equation}
Thus the candidate map $\widetilde F$ satisfies $P_1=P_2\circ\widetilde F$ on all of $\Omega_1$.
Since $C_1\subset B$, we have $C_1\subset B\cup\Omega_1$.
Equations~\eqref{eq:34} and~\eqref{eq:42} give
\begin{equation*}
  P_1(B\cup\Omega_1)
  =P_1(B)\cup P_1(\Omega_1)
  \subset P_2(M).
\end{equation*}
Hence $B\cup\Omega_1$ satisfies the same two conditions used to define the maximal open set $B$.
The maximality of $B$ therefore implies $B\cup\Omega_1\subset B$.
This is impossible because $x_1\in\Omega_1\setminus B$.
Hence $B$ is both open and closed, and connectedness gives $B=M$.

Equation~\eqref{eq:33} is equivalent to $P_2=P_1\circ F_0^{-1}$ on $C_2$.
Applying the preceding argument for the maximal open set with the indices $1$ and $2$ interchanged and with $F_0^{-1}:C_2\to C_1$ as the initial map gives $P_2(M)\subset P_1(M)$.
The equality $B=M$ already gives $P_1(M)\subset P_2(M)$ through \eqref{eq:34}, and hence $P_1(M)=P_2(M)$.
The injectivity of the Poisson embeddings in Lemma~\ref{lem:poisson-geometry} therefore makes the map $F=P_2^{-1}\circ P_1$ a bijection, and \eqref{eq:36} holds on $M$ because $B=M$.
The invertibility of $DF$ on $M^\circ$, together with $F=F_0$ on $C_1$ and the diffeomorphism property of $F_0$ in Proposition~\ref{prop:common-collar}, shows that $F$ is a global diffeomorphism.
Lemma~\ref{lem:poisson-geometry} gives the $\mathcal E^{\{\mathbf W\}}$ regularity of $F=P_2^{-1}\circ P_1$ and $F^{-1}=P_1^{-1}\circ P_2$ in the interior.
Equations~\eqref{eq:33} and~\eqref{eq:35} give $F=F_0$ on $C_1$ and hence $F^{-1}=F_0^{-1}$ on $C_2$, so Proposition~\ref{prop:common-collar} gives $F,F^{-1}\in\mathcal E^{\{\mathbf W\}}$ up to the boundary and yields $F|_{\partial M}=\operatorname{Id}$.

Finally, suppose that another diffeomorphism $\widehat F:M\to M$ satisfies $u_1^f=u_2^f\circ\widehat F$ on $M$ for every $f\in C^\infty(\partial M)$.
Definition~\ref{def:poisson-embedding} then gives $P_1=P_2\circ\widehat F$, and the injectivity of $P_2$ yields $\widehat F=P_2^{-1}\circ P_1=F$.

The proof is complete.
\end{proof}

The following rigidity lemma recovers the metric from global harmonic functions with boundary values supported in $\Gamma$.

\begin{lemma}\label{lem:metric-from-harmonic}
Let $\Gamma\subset\partial M$ be nonempty and relatively open, and let
\[
  F:M^\Gamma\longrightarrow M^\Gamma
\]
be a locally diffeomorphic $C^\infty$ map such that
\[
  u_1^f=u_2^f\circ F
  \quad\text{on }M^\Gamma,
  \qquad f\in C_c^\infty(\Gamma).
\]
Then there is a positive constant $\lambda$ such that
\[
  g_1=\lambda F^*g_2
  \quad\text{on }M^\Gamma.
\]
If $F^*g_2=g_1$ on a nonempty open subset of $M^\Gamma$, then $F^*g_2=g_1$ throughout $M^\Gamma$.
\end{lemma}

\begin{proof}
A locally diffeomorphic $C^\infty$ map preserving all harmonic extensions of boundary values supported in $\Gamma$ is a homothety when $n\ge3$~\cite[Proposition~4.1]{lassas2020poisson}.
Hence there is a positive constant $\lambda$ such that
$g_1=\lambda F^*g_2
  \quad\text{on }M^\Gamma.$
If $F^*g_2=g_1$ on a nonempty open set, comparison with $g_1=\lambda F^*g_2$ there gives $\lambda=1$.
\end{proof}

\begin{proof}[Proof of Theorem~\ref{thm:global-quasianalytic}]
The identity~\eqref{eq:3} gives the localized equality of DtN maps required by Lemma~\ref{lem:local-boundary-jets}, and hence equality of the full boundary jets in the respective boundary normal coordinates.
Proposition~\ref{prop:common-collar} converts this equality into a collar isometry $F_0$ and identifies the corresponding harmonic extensions there.
Lemmas~\ref{lem:runge-approximation}--\ref{lem:poisson-geometry} and Proposition~\ref{prop:poisson-globalization} extend the candidate map $P_2^{-1}\circ P_1$ to a boundary-fixing $\mathcal E^{\{\mathbf W\}}$ diffeomorphism $F:M\to M$.
Applying Lemma~\ref{lem:metric-from-harmonic} with $\Gamma=\partial M$, using $F=F_0$ and $F_0^*g_2=g_1$ on the collar $C_1$, finally gives $F^*g_2=g_1$ on $M$.
\end{proof}

\subsection{Uniqueness from partial boundary data}\label{subsec:global-partial-data}

Let $\Gamma\subset\partial M$ be nonempty and relatively open, and assume the partial DtN identity~\eqref{eq:local-quasianalytic-data}.
Recall that $M^\circ=M\setminus\partial M$ and $M^\Gamma=M^\circ\cup\Gamma$.
For $j=1,2$, define the partial Poisson embedding by
\begin{equation}
  P_j^\Gamma:M^\Gamma\longrightarrow\mathcal D'(\Gamma),
  \qquad
  P_j^\Gamma(x)(f)=u_j^f(x),
  \quad f\in C_c^\infty(\Gamma).
  \label{eq:partial-poisson}
\end{equation}
At a point $x\in\Gamma$, one has $P_j^\Gamma(x)=\delta_x$.
We shall construct an $\mathcal E^{\{\mathbf W\}}$ diffeomorphism $F:M^\Gamma\to M^\Gamma$ such that $F|_\Gamma=\operatorname{Id}$ and $F^*g_2=g_1$, and then extend $F|_{M^\circ}$ uniquely to the metric completion.
The construction first recovers a local isometry near $\Gamma$ and then extends it through $M^\circ$ using the partial Poisson embeddings; the completion extension determines the inaccessible boundary without fixing it pointwise.

The following lemma provides the local collar isometry that seeds the construction of the candidate map $(P_2^\Gamma)^{-1}\circ P_1^\Gamma$.

\begin{lemma}\label{lem:partial-common-collar}
Assume \eqref{eq:local-quasianalytic-data} holds.
For every $q\in\Gamma$, there are boundary collar neighborhoods $C_{j,q}$ of $q$ satisfying $C_{j,q}\cap\partial M\subset\Gamma$ and a diffeomorphism in the quasianalytic class $\mathcal{E}^{\{\mathbf{W}\}}$,
\begin{equation}
  F_{0,q}:C_{1,q}\longrightarrow C_{2,q},
  \qquad
  F_{0,q}|_{C_{1,q}\cap\partial M}=\operatorname{Id},
  \qquad
  F_{0,q}^*g_2=g_1,
  \label{eq:partial-common-collar}
\end{equation}
such that
\begin{equation}
  u_1^f=u_2^f\circ F_{0,q}
  \quad\text{on }C_{1,q}\cap M^\circ,
  \qquad f\in C_c^\infty(\Gamma).
  \label{eq:partial-collar-harmonic}
\end{equation}
\end{lemma}

\begin{proof}
Choose $\chi\in C_c^\infty(\Gamma)$ with $\chi=1$ near $q$.
For every $h\in C^\infty(\partial M)$, the function $\chi h$ belongs to $C_c^\infty(\Gamma)$, so \eqref{eq:local-quasianalytic-data}, followed by multiplication of the output by $\chi$, gives the operator identity
$\chi\Lambda_{g_1}\chi=\chi\Lambda_{g_2}\chi.$
The proof of Proposition~\ref{prop:common-collar} is local up to its final patching step.
Apply that local construction in a boundary chart $V_q\Subset\Gamma$ containing $q$.
Lemma~\ref{lem:local-boundary-jets} gives equality of the full boundary jets in the two boundary normal coordinate systems.
Lemma~\ref{lem:dc-stability}\textup{(iv)} gives two-sided boundary normal maps $\Phi_{j,q}$ of class $\mathcal E^{\{\mathbf W\}}$, so in these coordinates the differences of the corresponding metric coefficients belong to this class and have vanishing Taylor series along $V_q\times\{0\}$.
The quasianalytic identity principle in Lemma~\ref{lem:dc-stability}\textup{(iii)} therefore makes the metric coefficients equal on a neighborhood of $V_q\times\{0\}$.
After restricting to $t\ge0$, set
\[
  C_{j,q}=\Phi_{j,q}\bigl(V_q\times[0,\varepsilon_q)\bigr),
  \qquad
  F_{0,q}=\Phi_{2,q}\circ\Phi_{1,q}^{-1}.
\]
This gives \eqref{eq:partial-common-collar}.

Fix $f\in C_c^\infty(\Gamma)$ and set
$w^f=u_1^f-u_2^f\circ F_{0,q}.$
Since $F_{0,q}$ is an isometry, $w^f$ is $g_1$-harmonic on $C_{1,q}\cap M^\circ$.
Since $F_{0,q}$ fixes $V_q$ and maps the $g_1$-unit normal to the $g_2$-unit normal, the common Dirichlet trace and \eqref{eq:local-quasianalytic-data} give
\[
  w^f|_{V_q}=0,
  \qquad
  \partial_{\nu_{g_1}}w^f|_{V_q}
  =\Lambda_{g_1}^\Gamma f-\Lambda_{g_2}^\Gamma f=0.
\]
The boundary Cauchy uniqueness argument at the end of the proof of Proposition~\ref{prop:common-collar}, applied in the fixed two-sided coordinate neighborhood, gives $w^f=0$ on $C_{1,q}\cap M^\circ$.
This is \eqref{eq:partial-collar-harmonic}.
\end{proof}

The following proposition contains the second quasianalytic continuation step.

\begin{proposition}\label{prop:partial-poisson-globalization}
There exists a unique diffeomorphism
\[
  F:M^\Gamma\longrightarrow M^\Gamma
\]
in the quasianalytic class $\mathcal E^{\{\mathbf W\}}$ such that
\begin{equation}
  F|_\Gamma=\operatorname{Id},
  \qquad
  P_1^\Gamma=P_2^\Gamma\circ F.
  \label{eq:partial-global-poisson}
\end{equation}
Equivalently,
\begin{equation}
  u_1^f=u_2^f\circ F
  \quad\text{on }M^\Gamma,
  \qquad f\in C_c^\infty(\Gamma).
  \label{eq:partial-global-harmonic}
\end{equation}
\end{proposition}

\begin{proof}
The standard properties of the partial Poisson embedding show that each $P_j^\Gamma$ is smooth, injective, and immersive~\cite[Propositions~2.1 and~2.3]{lassas2020poisson}.
By~\cite[Lemma~2.4]{lassas2020poisson}, every locally defined map
$(P_j^\Gamma)^{-1}\circ P_i^\Gamma$
between open subsets of $M^\circ$ is a $C^\infty$ diffeomorphism onto its image.
At an interior point, Lemma~\ref{lem:runge-jets} supplies finitely many harmonic functions that form coordinates of class $\mathcal E^{\{\mathbf W\}}$ and have boundary values supported in $\Gamma$.
Using the local representation in the proof of Lemma~\ref{lem:poisson-geometry}, the elliptic regularity in Lemma~\ref{lem:dc-stability}\textup{(ii)} and the closure under local inversion and composition in Lemma~\ref{lem:dc-stability}\textup{(i)} then show that these maps belong to $\mathcal E^{\{\mathbf W\}}$.

Fix $q_0\in\Gamma$ and a connected interior portion of the collar in Lemma~\ref{lem:partial-common-collar}.
Let $B$ be the largest connected open subset of $M^\circ$ containing this seed and satisfying
$P_1^\Gamma(B)\subset P_2^\Gamma(M^\circ).$
On $B$, define
$F=(P_2^\Gamma)^{-1}\circ P_1^\Gamma.$
The collar identity \eqref{eq:partial-collar-harmonic} gives $F=F_{0,q_0}$ on the seed, and the preceding local representation shows that $F$ is a quasianalytic local diffeomorphism satisfying \eqref{eq:partial-global-harmonic} on $B$.

Suppose that $B\ne M^\circ$.
Since $M^\circ$ is connected, there are $p_k\in B$ and $x_1\in\partial_{M^\circ}B$ such that $p_k\to x_1$.
After passing to a subsequence, compactness of $M$ gives $F(p_k)\to x_2\in M$.
The limit cannot lie on $\partial M$: otherwise \eqref{eq:partial-global-harmonic} and boundary continuity would give
\[
  u_1^f(x_1)=\widetilde f(x_2)
  \qquad(f\in C_c^\infty(\Gamma)),
\]
which contradicts the strong maximum principle after choosing a nonzero $f\ge0$ with $\widetilde f(x_2)=0$; such a choice is possible whether or not $x_2$ belongs to $\Gamma$.
Thus $x_2\in M^\circ$.

The continuation argument in Proposition~\ref{prop:poisson-globalization} now applies with $P_j$ replaced by $P_j^\Gamma$.
Using boundary values supported in $\Gamma$, Lemma~\ref{lem:runge-jets} supplies the required harmonic coordinates and prescribed differentials, so the construction using these harmonic coordinates gives a local diffeomorphism $\widetilde F$ near $x_1$.
For every $f\in C_c^\infty(\Gamma)$, the coordinate representatives of $u_1^f$ and $u_2^f$ belong to $\mathcal E^{\{\mathbf W\}}$ and agree on a nonempty open subset; the quasianalytic identity principle in Lemma~\ref{lem:dc-stability}\textup{(iii)} therefore extends their equality across the connected coordinate neighborhood.
Hence $P_1^\Gamma=P_2^\Gamma\circ\widetilde F$ there, producing a connected open enlargement of $B$ with the same inclusion of Poisson images and contradicting maximality.
Consequently, $B=M^\circ$.

Interchanging $g_1$ and $g_2$ gives
$P_2^\Gamma(M^\circ)\subset P_1^\Gamma(M^\circ).$
The injectivity of the Poisson embeddings then makes $F:M^\circ\to M^\circ$ a global quasianalytic diffeomorphism satisfying \eqref{eq:partial-global-poisson}.
For any $q\in\Gamma$, equations \eqref{eq:partial-collar-harmonic} and \eqref{eq:partial-global-harmonic}, together with the injectivity of $P_2^\Gamma$, give
$F=F_{0,q}$
on the interior portion of a neighborhood of $q$.
Defining $F(q)=q$ therefore extends $F$ to a diffeomorphism $M^\Gamma\to M^\Gamma$ in the quasianalytic class and proves \eqref{eq:partial-global-poisson}--\eqref{eq:partial-global-harmonic} up to $\Gamma$.
Finally, any other map with \eqref{eq:partial-global-poisson} equals $(P_2^\Gamma)^{-1}\circ P_1^\Gamma$ by injectivity, which proves uniqueness.
\end{proof}

We now recover the metric using Lemma~\ref{lem:metric-from-harmonic}.

\begin{lemma}\label{lem:partial-metric-from-harmonic}
The diffeomorphism in Proposition~\ref{prop:partial-poisson-globalization} satisfies
\[
  F^*g_2=g_1\qquad\text{on }M^\Gamma.
\]
\end{lemma}

\begin{proof}
Equation~\eqref{eq:partial-global-harmonic} gives the harmonic identities required by Lemma~\ref{lem:metric-from-harmonic} for all boundary values supported in $\Gamma$.
On the nonempty interior portion of the seed collar, Proposition~\ref{prop:partial-poisson-globalization} gives $F=F_{0,q_0}$, while Lemma~\ref{lem:partial-common-collar} gives $F_{0,q_0}^*g_2=g_1$.
The conclusion follows from Lemma~\ref{lem:metric-from-harmonic}.
\end{proof}

We next determine what the interior isometry recovers on the inaccessible boundary by extending it to the metric completions.

\begin{proposition}\label{prop:partial-boundary-completion}
Let $F:M^\Gamma\to M^\Gamma$ be the isometry obtained in Lemma~\ref{lem:partial-metric-from-harmonic}.
Then $F|_{M^\circ}$ extends uniquely to an isometry
\[
  \overline F:(M,d_{g_1})\longrightarrow(M,d_{g_2}).
\]
Here $d_{g_j}$ is the Riemannian length distance defined in \eqref{eq:riemannian-length-distance}.
The extension agrees with $F$ on $M^\Gamma$, fixes $\Gamma$ pointwise, and maps $\partial M\setminus\Gamma$ onto itself.
\end{proposition}

\begin{proof}
For $x,y\in M^\circ$, let $c$ range over the piecewise $C^1$ curves in $M^\circ$ joining $x$ to $y$.
Define the intrinsic length distance of $(M^\circ,g_j)$ by
\[
  d_j^\circ(x,y):=\inf_c L_{g_j}(c).
\]
We first compare it with the Riemannian length distance $d_{g_j}$ on $M$ defined in \eqref{eq:riemannian-length-distance}.
Since every curve in $M^\circ$ is also a curve in $M$, one has
$d_{g_j}(x,y)\le d_j^\circ(x,y),
  \qquad x,y\in M^\circ.$
For the reverse inequality, let $c$ be a piecewise $C^1$ curve in $M$ with endpoints in $M^\circ$.
Choose a smooth collar $\kappa:\partial M\times[0,\delta)\to M$ so thin that the endpoints of $c$ lie outside its image, and choose $\chi\in C_c^\infty([0,\delta))$ with $\chi=1$ near $0$.
For sufficiently small $\varepsilon>0$, define $r_\varepsilon:M\to M^\circ$ by
$r_\varepsilon(\kappa(q,t))=\kappa(q,t+\varepsilon\chi(t))$
in the collar and by the identity outside it.
Then $r_\varepsilon$ fixes the endpoints of $c$ and converges to the identity in $C^1$.
Consequently, the interior curves $c_\varepsilon=r_\varepsilon\circ c$ satisfy
$L_{g_j}(c_\varepsilon)\longrightarrow L_{g_j}(c)
  \qquad(\varepsilon\to0).$
Taking the infimum first over such interior perturbations and then over $c$ gives
\begin{equation}
  d_j^\circ(x,y)=d_{g_j}(x,y),
  \qquad x,y\in M^\circ.
  \label{eq:interior-completion-distance}
\end{equation}
The space $(M,d_{g_j})$ is complete because it is compact, and $M^\circ$ is dense in $M$.
Thus \eqref{eq:interior-completion-distance} identifies $(M,d_{g_j})$ canonically with the metric completion of $(M^\circ,d_j^\circ)$.

By Proposition~\ref{prop:partial-poisson-globalization}, $F|_{M^\circ}$ is a diffeomorphism of $M^\circ$, and Lemma~\ref{lem:partial-metric-from-harmonic} gives $F^*g_2=g_1$ there.
It therefore preserves lengths of piecewise $C^1$ curves; applying the same observation to $F^{-1}$ gives
$d_2^\circ(F(x),F(y))=d_1^\circ(x,y),
  \qquad x,y\in M^\circ.$
For $y\in M$, choose a sequence $x_k\in M^\circ$ converging to $y$ in $d_{g_1}$.
Equations above show that $(F(x_k))$ is Cauchy in $(M,d_{g_2})$; define
\[
  \overline F(y)=\lim_{k\to\infty}F(x_k).
\]
If $(x_k')$ is another interior sequence converging to $y$, then
$d_{g_2}(F(x_k),F(x_k'))=d_{g_1}(x_k,x_k')\longrightarrow0,$
so the definition is independent of the approximating sequence.
For $y,z\in M$, taking interior sequences $x_k\to y$ and $z_k\to z$ yields
\[
  d_{g_2}(\overline F(y),\overline F(z))
  =\lim_{k\to\infty}d_{g_2}(F(x_k),F(z_k))
  =d_{g_1}(y,z).
\]
Hence $\overline F$ is an isometric extension of $F|_{M^\circ}$.

Applying the same construction to $F^{-1}$ gives an isometric extension $\overline G:(M,d_{g_2})\to(M,d_{g_1})$.
The compositions $\overline G\circ\overline F$ and $\overline F\circ\overline G$ are the identity on the dense set $M^\circ$, and hence everywhere by continuity.
Thus $\overline F$ is surjective with inverse $\overline G$.
It is also the unique extension: two isometric extensions agreeing with $F$ on the dense subset $M^\circ$ agree on all of $M$.

We next determine the boundary behavior.
The map $\overline F$ restricts to $F$ on $M^\circ$, and its inverse restricts to $F^{-1}$ there.
If $y\in\partial M$ and $\overline F(y)\in M^\circ$, then
\[
  y=\overline F^{-1}(\overline F(y))=F^{-1}(\overline F(y))\in M^\circ,
\]
a contradiction.
The same argument for $\overline F^{-1}$ shows that
$\overline F(\partial M)=\partial M.$

Finally, let $q\in\Gamma$ and choose $x_k\in M^\circ$ with $x_k\to q$.
Proposition~\ref{prop:partial-poisson-globalization} gives continuity of $F$ on $M^\Gamma$ and $F(q)=q$, so
\[
  \overline F(q)=\lim_{k\to\infty}F(x_k)=F(q)=q.
\]
Thus $\overline F$ agrees with $F$ on $M^\Gamma$ and fixes $\Gamma$ pointwise.
Since $\overline F$ is injective, no point of $\partial M\setminus\Gamma$ can map into the pointwise fixed set $\Gamma$.
Together with $\overline F(\partial M)=\partial M$ and surjectivity, this proves
$\overline F(\partial M\setminus\Gamma)=\partial M\setminus\Gamma.$
\end{proof}

Proposition~\ref{prop:partial-boundary-completion} gives the precise boundary conclusion of the partial-data problem: the accessible boundary is fixed pointwise, whereas the inaccessible boundary is determined only as the completion boundary of the interior isometry and need not be fixed pointwise.

\begin{proof}[Proof of Corollary~\ref{cor:global-quasianalytic-local}]
The identity~\eqref{eq:local-quasianalytic-data} gives the localized equality of DtN maps used in Lemma~\ref{lem:partial-common-collar}, which converts it into collar isometries of class $\mathcal E^{\{\mathbf W\}}$ near the accessible boundary and identifies there the harmonic extensions of all boundary values supported in $\Gamma$.
Lemma~\ref{lem:runge-jets} and Proposition~\ref{prop:partial-poisson-globalization} extend the candidate map $(P_2^\Gamma)^{-1}\circ P_1^\Gamma$ to an $\mathcal E^{\{\mathbf W\}}$ diffeomorphism $F:M^\Gamma\to M^\Gamma$ fixing $\Gamma$.
Applying Lemma~\ref{lem:partial-metric-from-harmonic}, using the seed collar isometry from Lemma~\ref{lem:partial-common-collar}, gives $F^*g_2=g_1$ on $M^\Gamma$.
Proposition~\ref{prop:partial-boundary-completion} finally extends $F|_{M^\circ}$ uniquely to the asserted isometry between the metric completions, which fixes $\Gamma$ pointwise and maps $\partial M\setminus\Gamma$ onto itself.
\end{proof}

\section{Uniqueness with quasianalyticity in the normal variable}\label{sec:proof-normal}

Fix a one-dimensional quasianalytic weight sequence $\mathbf W$ as in Definition~\ref{def:one-dimensional-quasianalytic-weight}.
Let $n\ge3$, let $L>0$, let $Y$ be a closed connected smooth $(n-1)$-manifold, and let $M=[0,L]\times Y$ have boundary components $\Gamma_0$ and $\Gamma_L$ as in~\eqref{eq:9}.
Throughout this section, let
\begin{equation}
  g_j=\,\mathrm{d}t^2+h_j(t,y)\in\mathcal Q_{\mathbf W}(M),
  \qquad j=1,2,
  \label{eq:normal-proof-metrics}
\end{equation}
and assume that their DtN maps satisfy~\eqref{eq:13}.

The proof of Theorem~\ref{thm:normal-quasianalytic} has three steps.
First, the prescribed product form makes $(t,y)$ a common global boundary normal coordinate system for the two metrics.
Second, smooth boundary determination shows that their full normal jets at $\Gamma_0$ agree.
Finally, the estimate~\eqref{eq:12} and the Denjoy--Carleman condition~\eqref{eq:10} propagate this jet equality along each normal curve $t\mapsto(t,y)$.
Thus quasianalytic continuation is used only in the normal variable, while the tangential dependence remains smooth.

\begin{lemma}\label{lem:common-normal-distance}
Let $g_1$ and $g_2$ be the metrics in Theorem~\ref{thm:normal-quasianalytic}, written in the prescribed product coordinates as in~\eqref{eq:normal-proof-metrics}.
Let $d_{g_j}$ be the Riemannian length distance defined in~\eqref{eq:riemannian-length-distance}.
The $g_j$-distance from the boundary component $\Gamma_0$ to a point $x\in M$ is defined by
\[
  d_{g_j}(\Gamma_0,x):=\inf_{z\in\Gamma_0}d_{g_j}(z,x).
\]
Then
\[
  d_{g_j}\bigl(\Gamma_0,(t,y)\bigr)=t
  \qquad\bigl((t,y)\in M,\ j=1,2\bigr).
\]
Consequently, the prescribed product coordinate system $(t,y)$ is a common global boundary normal coordinate system for $g_1$ and $g_2$.
\end{lemma}

\begin{proof}
For $j=1,2$, the identities $(g_j)_{tt}=1$ and $(g_j)_{t\alpha}=0$ in \eqref{eq:normal-proof-metrics} show that $\partial_t$ is the inward unit normal along $\Gamma_0$.
They also give $\Gamma_{tt}^i=0$, so, for each $y\in Y$, the curve $t\mapsto(t,y)$ is a unit-speed geodesic normal to $\Gamma_0$.

Fix $(t_0,y_0)\in M$, and let $c:[0,1]\to M$, $c(s)=(t(s),y(s))$, be any piecewise $C^1$ curve such that $c(0)\in\Gamma_0$ and $c(1)=(t_0,y_0)$.
Then
\begin{equation}
  L_{g_j}(c)
  =\int_0^1\sqrt{\dot t(s)^2+h_j(t(s),y(s))(\dot y(s),\dot y(s))}\,\mathrm{d}s
  \ge\int_0^1|\dot t(s)|\,\mathrm{d}s
  \ge |t(1)-t(0)|=t_0.
  \label{eq:53}
\end{equation}
The vertical curve $c_0(s)=(st_0,y_0)$ joins $\Gamma_0$ to $(t_0,y_0)$ and satisfies $L_{g_j}(c_0)=t_0$.
Taking the infimum over $c$ therefore gives
$d_{g_j}\bigl(\Gamma_0,(t_0,y_0)\bigr)=t_0.$
Finally, uniqueness for the geodesic equation gives
$\exp_{(0,y)}^{g_j}(t\partial_t)=(t,y),
  \qquad 0\le t\le L,$
and this normal exponential parametrization covers all of $M$ in the prescribed product coordinates.
It is the same for $j=1,2$, which proves the final assertion.
\end{proof}

\begin{lemma}\label{lem:normal-quasianalytic-continuation}
Let $\mathbf W$ be as in Definition~\ref{def:one-dimensional-quasianalytic-weight}, and let $g_1,g_2$ belong to the normal-variable quasianalytic metric class $\mathcal Q_{\mathbf W}(M)$ of Definition~\ref{def:normal-variable-quasianalytic-class}.
Write the metrics as in~\eqref{eq:normal-proof-metrics}.
Fix $y_*\in Y$.
If
\[
  \partial_t^k h_1(0,y_*)=\partial_t^k h_2(0,y_*)
  \qquad(k=0,1,2,\ldots),
\]
then
\[
  h_1(t,y_*)=h_2(t,y_*)
  \qquad(0\le t\le L).
\]
\end{lemma}

\begin{proof}
Let $\widetilde h_j$ be the extension supplied by Definition~\ref{def:normal-variable-quasianalytic-class}, defined on $(-\varepsilon_j,L+\varepsilon_j)\times Y$, and set $\varepsilon:=\min\{\varepsilon_1,\varepsilon_2\}$.
Choose a coordinate chart about $y_*$ and fix component indices $\mu,\nu$.
On $I:=(-\varepsilon,L+\varepsilon)$, set
$\delta h_{\mu\nu,y_*}(t)
  :=(\widetilde h_1)_{\mu\nu}(t,y_*)
    -(\widetilde h_2)_{\mu\nu}(t,y_*).$
Estimate~\eqref{eq:12}, after enlarging the constants, gives the one-dimensional derivative bounds of Roumieu type with weight $\mathbf W$ for $\delta h_{\mu\nu,y_*}$ on every compact subinterval of $I$.
The hypothesis gives $\partial_t^k\delta h_{\mu\nu,y_*}(0)=0$ for every $k\ge0$.
By the one-dimensional Denjoy--Carleman theorem, the logarithmic convexity of $\mathbf W$ and condition~\eqref{eq:10} make the Taylor map injective on this class~\cite[Section~2.7]{rainer2014composition}; hence $\delta h_{\mu\nu,y_*}=0$ on $I$.
Since the component indices were arbitrary, the asserted tensor equality follows on $[0,L]$.
\end{proof}

\begin{proposition}\label{prop:normal-boundary-continuation}
Let $\mathbf W$ be as in Definition~\ref{def:one-dimensional-quasianalytic-weight}, and let $g_1,g_2\in\mathcal Q_{\mathbf W}(M)$ be the metrics in Theorem~\ref{thm:normal-quasianalytic}, written as in~\eqref{eq:normal-proof-metrics}.
Here $\mathcal Q_{\mathbf W}(M)$ is the normal-variable quasianalytic metric class from Definition~\ref{def:normal-variable-quasianalytic-class}.
If their DtN maps agree as in~\eqref{eq:13}, then
\[
  h_1=h_2\qquad\text{on }[0,L]\times Y.
\]
\end{proposition}

\begin{proof}
Fix $q=(0,y_0)\in\Gamma_0$, and choose $\chi\in C_c^\infty(\partial M)$ supported in $\Gamma_0$ and equal to one near $q$.
Equation~\eqref{eq:13} implies
$\chi\Lambda_{g_1}\chi=\chi\Lambda_{g_2}\chi.$
By Lemma~\ref{lem:common-normal-distance}, the prescribed product coordinates are common boundary normal coordinates for the two metrics.
Lemma~\ref{lem:local-boundary-jets}, applied in a product boundary chart about $q$ with the same tangential coordinates $y$ for both metrics, therefore gives exactly the conclusion~\eqref{eq:local-boundary-jets}:
\[
  \partial_t^k h_1(0,y)=\partial_t^k h_2(0,y)
  \qquad(k=0,1,2,\ldots)
\]
near $y_0$.
Since $q\in\Gamma_0$ was arbitrary, their full normal jets agree on all of $\Gamma_0$.

For each $y_*\in Y$, Lemma~\ref{lem:normal-quasianalytic-continuation} now gives $h_1(t,y_*)=h_2(t,y_*)$ for all $t\in[0,L]$.
Since $y_*$ was arbitrary, $h_1=h_2$ on $[0,L]\times Y$.
\end{proof}

\begin{proof}[Proof of Theorem~\ref{thm:normal-quasianalytic}]
Lemma~\ref{lem:common-normal-distance} identifies the prescribed product coordinates as common global boundary normal coordinates for the two metrics.
The identity~\eqref{eq:13} then gives the localized equality of DtN maps required by Lemma~\ref{lem:local-boundary-jets}, and hence equality of the full normal jets of $h_1$ and $h_2$ on $\Gamma_0$.
Lemma~\ref{lem:normal-quasianalytic-continuation} propagates this equality along every normal curve $t\mapsto(t,y)$, and Proposition~\ref{prop:normal-boundary-continuation} gives $h_1=h_2$ on $[0,L]\times Y$.
Together with~\eqref{eq:normal-proof-metrics}, this proves $g_1=g_2$ on $M$, as asserted in~\eqref{eq:14}.
This argument uses only the DtN observations on $\Gamma_0$; the corresponding assertion for $\Gamma_L$ follows by replacing $t$ with $L-t$.
\end{proof}

\begin{proof}[Proof of Corollary~\ref{cor:product-base}]
The common base metric and absence of mixed terms give common local boundary normal coordinates, with the fiber coordinate $y$ unchanged.
Lemma~\ref{lem:local-boundary-jets}, localized on $A$, yields equality of the normal jets of $h_1$ and $h_2$.
Tangential differentiation shows that the full base-variable jet of $h_1-h_2$ vanishes at each $x_0\in\Sigma$, for every fixed $y\in M_2$.
The vanishing of this jet is preserved by the smooth change back to the original quasianalytic base coordinates.
In these coordinates, the assumed quasianalytic extensions across $\partial M_1$ allow the identity principle to give $h_1=h_2$ near $x_0$ and then throughout the connected interior $M_1^\circ$.
Continuity yields $g_1=g_2$ on $M$.
\end{proof}

\section{Smooth uniqueness under tangential homogeneity}\label{sec:proof-homogeneous}

This section is devoted to the proof of Theorem~\ref{thm:tangentially-homogeneous}.
Let $d\ge2$, let $L>0$, and set $M=[0,L]\times\mathbb{T}^d$, with the prescribed product coordinates $(t,y)$.
Fix $g_1,g_2$ in the tangentially homogeneous metric class $\mathcal S(M)$ of Definition~\ref{def:tangentially-homogeneous-class}, and assume that their DtN maps satisfy~\eqref{eq:tangentially-homogeneous-data}.
Write $h_j(t)$ for the tangential coefficient matrix of $g_j$ in~\eqref{eq:15}, $j=1,2$.

Fourier decomposition and the Liouville transformation reduce the boundary data to discrete samples of one-dimensional Weyl functions (Subsection~\ref{subsec:homogeneous-fourier}).
A discrete inverse spectral argument then determines the corresponding potentials and Liouville factors (Subsection~\ref{subsec:homogeneous-inverse-spectral}).
Finally, Subsection~\ref{subsec:homogeneous-metric-recovery} synchronizes the Liouville coordinates for different lattice directions and recovers the metric in the prescribed normal coordinate $t$.

\subsection{Fourier reduction}\label{subsec:homogeneous-fourier}

We first use tangential homogeneity to separate variables in the harmonic equation and identify the corresponding Fourier blocks of the DtN map.

\begin{lemma}\label{lem:fourier-density-flux}
Fix $j\in\{1,2\}$, and let $g_j$ belong to the tangentially homogeneous metric class $\mathcal S(M)$ of Definition~\ref{def:tangentially-homogeneous-class}.
Write $\omega_j(t)=\sqrt{\det h_j(t)}$.
For each $k\in\mathbb{Z}^d$, a smooth function $u(t)e^{ik\cdot y}$ is $g_j$-harmonic if and only if
\begin{equation}
  (\omega_ju_t)_t-\omega_jk^Th_j^{-1}ku=0,
  \qquad 0<t<L.
  \label{eq:64}
\end{equation}
For this mode, $\Lambda_{g_j}$ maps the endpoint values $(u(0),u(L))$ to the outward normal derivatives $(-u_t(0),u_t(L))$.
After weighting by the endpoint densities, this linear map is represented by the $2\times2$ weighted endpoint flux matrix $D_j(k)$:
\begin{equation}
  \begin{pmatrix}
    -\omega_j(0)u_t(0)\\
    \omega_j(L)u_t(L)
  \end{pmatrix}
  =D_j(k)\begin{pmatrix}u(0)\\u(L)\end{pmatrix},
  \qquad
  D_j(k)=
  \begin{pmatrix}
    D_j^{00}(k)&D_j^{0L}(k)\\
    D_j^{L0}(k)&D_j^{LL}(k)
  \end{pmatrix}.
  \label{eq:65}
\end{equation}
The row and column indices indicate the output and input endpoints, respectively.
The matrix $D_j(k)$ is real and symmetric.
\end{lemma}

\begin{proof}
Since $h_j$ depends only on $t$, direct calculation gives
\[
  \Delta_{g_j}\bigl(u(t)e^{ik\cdot y}\bigr)
  =\left[\omega_j^{-1}(\omega_ju_t)_t
  -(k^Th_j^{-1}k)u\right]e^{ik\cdot y}.
\]
This proves the equivalence with~\eqref{eq:64}.

By uniqueness of the Dirichlet problem~\eqref{eq:2}, boundary data in the $k$-th Fourier mode give a solution in the same mode.
The outward unit normals at $t=0$ and $t=L$ are $-\partial_t$ and $\partial_t$, respectively, which gives the stated action of $\Lambda_{g_j}$.
Linearity then gives the matrix representation~\eqref{eq:65}.

For real-valued solutions $u$ and $v$ of~\eqref{eq:64}, integration by parts gives
\[
  -\omega_j(0)u_t(0)v(0)+\omega_j(L)u_t(L)v(L)
  =\int_0^L\omega_j\bigl(u_tv_t+(k^Th_j^{-1}k)uv\bigr)\,\mathrm{d}t.
\]
The right-hand side is symmetric in $u$ and $v$, so $D_j(k)$ in~\eqref{eq:65} is symmetric.
Its entries are real because~\eqref{eq:64} has real coefficients.
\end{proof}

We put the Fourier equation~\eqref{eq:64} in Schr\"odinger form.
Fix $\xi\in\mathbb{Z}^d\setminus\{0\}$ and consider the modes $k=N\xi$, $N=1,2,\ldots$.
For $j=1,2$, set $q_{j,\xi}(t)=\xi^Th_j(t)^{-1}\xi>0$.
The original unknown $u(t)$, the coefficient of $e^{iN\xi\cdot y}$ in a harmonic Fourier mode, satisfies~\eqref{eq:64} in the variable $t$:
\[
  -(\omega_ju_t)_t=-N^2\omega_jq_{j,\xi}u,
  \qquad 0<t<L.
\]

We introduce the new independent variable $s$ by
\begin{equation}
  s=s_{j,\xi}(t)=\int_0^t\sqrt{q_{j,\xi}(r)}\,\mathrm{d} r,
  \qquad S_{j,\xi}=s_{j,\xi}(L).
  \label{eq:66}
\end{equation}
Since $q_{j,\xi}>0$, the map $s_{j,\xi}$ is a smooth increasing diffeomorphism from $[0,L]$ onto $[0,S_{j,\xi}]$.
We then define the positive Liouville factor $\mu_{j,\xi}$ and replace $u(t)$ by the new unknown $v=v(s)$ through
\begin{equation}
  \begin{aligned}
    &\mu_{j,\xi}(s_{j,\xi}(t))=(\omega_j(t)^2q_{j,\xi}(t))^{1/4},\\
    &v(s_{j,\xi}(t))=\mu_{j,\xi}(s_{j,\xi}(t))u(t).
  \end{aligned}
  \label{eq:67}
\end{equation}
The change of variable~\eqref{eq:66} and the rescaling~\eqref{eq:67} together form the Liouville transformation~\cite{marchenko1986sturm}.

Under this transformation, $u$ solves~\eqref{eq:64} with $k=N\xi$ if and only if $v$ solves the Schr\"odinger equation
\begin{equation}
  -v''(s)+V_{j,\xi}(s)v(s)=-N^2v(s),
  \qquad V_{j,\xi}(s)=\frac{\mu_{j,\xi}''(s)}{\mu_{j,\xi}(s)},
  \qquad 0<s<S_{j,\xi}.
  \label{eq:68}
\end{equation}
Here $V_{j,\xi}$ is the transformed potential, $-N^2$ is the spectral parameter, and primes denote differentiation with respect to the new variable $s$.

To describe the response at the left endpoint, we consider solutions of~\eqref{eq:64} with $u(L)=0$.
By~\eqref{eq:67}, this is equivalent to the right Dirichlet condition $v(S_{j,\xi})=0$.
For real $N>0$, let $v$ be any nonzero solution of~\eqref{eq:68} satisfying this condition.
The associated Weyl function is defined by
\begin{equation}
  \mathfrak m_{j,\xi}(N)=-\frac{v'(0)}{v(0)}.
  \label{eq:directional-weyl-function}
\end{equation}
This quotient is well defined, as shown in the proof below.

The following proposition is the main intermediate result of this section.
For $j=1,2$ and each $\xi\in\mathbb{Z}^d\setminus\{0\}$, it determines from the DtN map $\Lambda_{g_j}$ the interval length $S_{j,\xi}$, the initial values $\mu_{j,\xi}(0)$ and $\mu'_{j,\xi}(0)$, and the Weyl samples $\mathfrak m_{j,\xi}(N)$ for $N=1,2,\ldots$.
These data are then used in Subsection~\ref{subsec:homogeneous-inverse-spectral} to recover the transformed potential $V_{j,\xi}$ and the Liouville factor $\mu_{j,\xi}$.

\begin{proposition}\label{prop:directional-weyl-data}
Let $g_1,g_2$ belong to the tangentially homogeneous metric class $\mathcal S(M)$ of Definition~\ref{def:tangentially-homogeneous-class}.
Assume that $\Lambda_{g_1}=\Lambda_{g_2}$, as in~\eqref{eq:tangentially-homogeneous-data}.
Then
\[
  h_1(0)=h_2(0),\qquad h_1(L)=h_2(L).
\]
Moreover, for every $\xi\in\mathbb{Z}^d\setminus\{0\}$, the quantities $S_{j,\xi}$, $\mu_{j,\xi}$, and $\mathfrak m_{j,\xi}$ defined in~\eqref{eq:66}, \eqref{eq:67}, and~\eqref{eq:directional-weyl-function}, respectively, satisfy
\[
  \begin{gathered}
    S_{1,\xi}=S_{2,\xi},\qquad
    \mu_{1,\xi}(0)=\mu_{2,\xi}(0),\qquad
    \mu'_{1,\xi}(0)=\mu'_{2,\xi}(0),\\
    \mathfrak m_{1,\xi}(N)=\mathfrak m_{2,\xi}(N),
    \qquad N=1,2,\ldots.
  \end{gathered}
\]
\end{proposition}

\begin{proof}
Fix $\xi\in\mathbb{Z}^d\setminus\{0\}$ and put $k=N\xi$.
The change of variable~\eqref{eq:66} and the rescaling~\eqref{eq:67} transform~\eqref{eq:64} into~\eqref{eq:68}.
Indeed, $\mathrm{d}s_{j,\xi}/\mathrm{d}t=\sqrt{q_{j,\xi}}$ and $\mu_{j,\xi}^2=\omega_j\sqrt{q_{j,\xi}}$, so direct differentiation gives the flux identity
\begin{equation}
  \omega_j(t)u_t(t)
  =\mu_{j,\xi}(s)v_s(s)-v(s)(\mu_{j,\xi})_s(s).
  \label{eq:69}
\end{equation}
Here $s=s_{j,\xi}(t)$, and $u_t$ denotes differentiation with respect to $t$.

We first verify that the Weyl function in~\eqref{eq:directional-weyl-function} is well defined for every $N>0$.
The nonzero solution $v$ in its definition already satisfies $v(S_{j,\xi})=0$, so we need to show that $v(0)\ne0$.
Suppose instead that $v(0)=0$.
By~\eqref{eq:67}, the corresponding solution $u$ of~\eqref{eq:64} with $k=N\xi$ then vanishes at both endpoints.
Multiplying the equation by $\overline u$, integrating by parts, and using the two zero boundary values gives
\begin{equation}
  \int_0^L\omega_j(t)
  \left(|u_t(t)|^2+N^2q_{j,\xi}(t)|u(t)|^2\right)\,\mathrm{d} t=0.
  \label{eq:weyl-nonresonance}
\end{equation}
Both coefficients are strictly positive, so $u\equiv0$.
Equation~\eqref{eq:67} then gives $v\equiv0$, contradicting the choice of $v$.
Thus $v(0)\ne0$.
Any two such solutions are proportional by uniqueness for the initial value problem at $s=S_{j,\xi}$, so the quotient defining $\mathfrak m_{j,\xi}(N)$ is independent of the choice of $v$.

Since $V_{j,\xi}$ is smooth, the classical high-energy asymptotics for the right Dirichlet problem~\cite[Corollary~4.3]{simon1999new} yield
\begin{equation}
  \mathfrak m_{j,\xi}(N)=N+O(N^{-1})\qquad(N\to\infty).
  \label{eq:70}
\end{equation}
Our definition~\eqref{eq:directional-weyl-function} has the opposite sign to the Weyl function used in that reference.

For the solution with $u(0)=1$, $u(L)=0$, \eqref{eq:69} gives
\[
  -u_t(0)=\sqrt{q_{j,\xi}(0)}
  \left(\mathfrak m_{j,\xi}(N)+\frac{\mu_{j,\xi}'(0)}{\mu_{j,\xi}(0)}\right).
\]
Consequently,
\begin{equation}
  \lim_{N\to\infty}\frac{-u_t(0)}N=\sqrt{\xi^Th_j(0)^{-1}\xi}.
  \label{eq:71}
\end{equation}
The same calculation after reversing $t$ gives the corresponding limit at $t=L$.
Let $e_1,\ldots,e_d$ be the standard basis of $\mathbb{Z}^d$ and set
\[
  \mathcal K
  =\{e_a:1\le a\le d\}
  \cup\{e_a+e_b:1\le a<b\le d\}.
\]
By~\eqref{eq:71}, the quadratic form of $h_j(0)^{-1}$ is determined on $\mathcal K$, which suffices to recover all its matrix entries.
The same argument at $t=L$ determines $h_j(L)^{-1}$.

Equality of the DtN maps therefore implies
\begin{equation}
  h_1(0)=h_2(0),\quad h_1(L)=h_2(L),\quad
  \omega_1(0)=\omega_2(0),\quad \omega_1(L)=\omega_2(L).
  \label{eq:72}
\end{equation}
The equalities for $\omega_j$ follow from those for $h_j$ and the definition $\omega_j(t)=\sqrt{\det h_j(t)}$ in Lemma~\ref{lem:fourier-density-flux}.
We may hence multiply the two equal DtN maps by the common endpoint density factors.
Thus the weighted endpoint flux matrices defined in~\eqref{eq:65} satisfy
$D_1(k)=D_2(k),\qquad k\in\mathbb{Z}^d.$

For $k=N\xi$, write $D_{j,\xi}(N)=D_j(N\xi)$, with entries indexed as in~\eqref{eq:65}.
In particular, $D_{j,\xi}^{00}(N)=-\omega_j(0)u_t(0)$ for the solution with $u(0)=1$ and $u(L)=0$.
Since $v(0)=\mu_{j,\xi}(0)$ when $u(0)=1$, formula \eqref{eq:69} yields
\begin{equation}
  D_{j,\xi}^{00}(N)
  =\mu_{j,\xi}(0)^2\mathfrak m_{j,\xi}(N)+\mu_{j,\xi}(0)\mu_{j,\xi}'(0).
  \label{eq:73}
\end{equation}
It follows from the common data and \eqref{eq:70} that
\begin{equation}
  \mu_{1,\xi}(0)^2=\mu_{2,\xi}(0)^2
  =\lim_{N\to\infty}\frac{D_{1,\xi}^{00}(N)}N,
  \label{eq:74}
\end{equation}
and
\begin{equation}
  \frac{\mu_{1,\xi}'(0)}{\mu_{1,\xi}(0)}
  =\frac{\mu_{2,\xi}'(0)}{\mu_{2,\xi}(0)}
  =\lim_{N\to\infty}
  \left(\frac{D_{1,\xi}^{00}(N)}{\mu_{1,\xi}(0)^2}-N\right).
  \label{eq:75}
\end{equation}
Subtracting the common constant term in \eqref{eq:73} then gives
\begin{equation}
  \mathfrak m_{1,\xi}(N)=\mathfrak m_{2,\xi}(N)\qquad(N=1,2,\ldots).
  \label{eq:76}
\end{equation}

The off-diagonal entry $D_{j,\xi}^{0L}(N)$ determines the Liouville interval length.
Let $\varphi_{j,\xi,N}$ solve
\[
  -\varphi_{j,\xi,N}''+V_{j,\xi}\varphi_{j,\xi,N}
  =-N^2\varphi_{j,\xi,N},
  \qquad \varphi_{j,\xi,N}(0)=0,
  \qquad \varphi_{j,\xi,N}'(0)=1,
  \qquad 0<s<S_{j,\xi}.
\]
For the solution with $u(0)=0$ and $u(L)=1$, equations~\eqref{eq:65} and~\eqref{eq:69} give
\begin{equation}
  D_{j,\xi}^{0L}(N)
  =-\frac{\mu_{j,\xi}(0)\mu_{j,\xi}(S_{j,\xi})}
  {\varphi_{j,\xi,N}(S_{j,\xi})}.
  \label{eq:77}
\end{equation}
Variation of constants gives the integral equation
\[
  \varphi_{j,\xi,N}(s)=\frac{\sinh(Ns)}N+
  \int_0^s\frac{\sinh(N(s-r))}N
  V_{j,\xi}(r)\varphi_{j,\xi,N}(r)\,\mathrm{d} r.
\]
After multiplication by $e^{-Ns}$, Gronwall's inequality gives $e^{-Ns}|\varphi_{j,\xi,N}(s)|=O(N^{-1})$ uniformly for $s\in[0,S_{j,\xi}]$ as $N\to\infty$.
Substitution into the integral term then yields
\begin{equation}
  \varphi_{j,\xi,N}(S_{j,\xi})
  =\frac{e^{NS_{j,\xi}}}{2N}(1+O(N^{-1})),
  \qquad
  \lim_{N\to\infty}\frac1N
  \log|\varphi_{j,\xi,N}(S_{j,\xi})|=S_{j,\xi}.
  \label{eq:78}
\end{equation}
In particular $\varphi_{j,\xi,N}(S_{j,\xi})$ is nonzero for all sufficiently large $N$.
Both endpoint factors in \eqref{eq:77} are already common by \eqref{eq:72} and \eqref{eq:67}.
Since $D_{1,\xi}^{0L}(N)=D_{2,\xi}^{0L}(N)$, \eqref{eq:77} and~\eqref{eq:78} imply
\begin{equation}
  S_{1,\xi}=S_{2,\xi}=:S_\xi.
  \label{eq:79}
\end{equation}

The proof is complete.
\end{proof}

\subsection{Inverse spectral uniqueness}\label{subsec:homogeneous-inverse-spectral}

We next show that the values of the Weyl function at sufficiently large integers determine the potential.
The proof combines the classical $A$-amplitude representation with a uniqueness result for discrete exponential moments.

\begin{lemma}\label{lem:a-amplitude}
Let $S>0$ and $V\in C^\infty([0,S];\mathbb{R})$.
For sufficiently large real $\kappa>0$, let $v$ be any nonzero solution of
\[
  -v''+Vv=-\kappa^2v,\qquad 0<s<S,\qquad v(S)=0.
\]
Define the Weyl function by
\[
  \mathfrak m(\kappa)=-\frac{v'(0)}{v(0)}.
\]
There is a function $\mathfrak a_V\in C([0,S);\mathbb{R})$, called the $A$-amplitude associated with $V$, such that, for every $0<b<S$ and every $\varepsilon>0$,
\begin{equation}
  \mathfrak m(\kappa)
  =\kappa+\int_0^b\mathfrak a_V(\alpha)e^{-2\alpha\kappa}\,\mathrm{d}\alpha
  +O(e^{-(2b-\varepsilon)\kappa})
  \qquad(\kappa\to+\infty).
  \label{eq:80}
\end{equation}
For two such potentials $V_1,V_2$ and every $0<a<S$,
\[
  \mathfrak a_{V_1}=\mathfrak a_{V_2}\ \text{on }[0,a]
  \quad\Longleftrightarrow\quad
  V_1=V_2\ \text{on }[0,a].
\]
\end{lemma}

\begin{proof}
For the finite interval and right Dirichlet condition considered here, the representation and local correspondence follow from~\cite[Theorems~1.4 and~1.5]{simon1999new}; see also~\cite{gesztesy2000new}.
Since $V$ is smooth, the amplitude is continuous.
Those references use $v'(0)/v(0)$, whose negative is our Weyl function; this gives the signs in~\eqref{eq:80}.
For related local Borg--Marchenko uniqueness results, see~\cite{gesztesy2000local,bennewitz2001proof}.
\end{proof}

The next lemma gives a uniqueness result for discrete exponential moments.

\begin{lemma}\label{lem:discrete-exponential-moments}
Let \(\ell>0\) and \(G\in C([0,\ell])\).
If, for some integer \(N_0\),
\begin{equation}
  \sup_{N\ge N_0}
  N\left|\int_0^\ell e^{Nr}G(r)\,\mathrm{d} r\right|<\infty,
  \qquad N\in\mathbb N,
  \label{eq:discrete-moment-assumption}
\end{equation}
then \(G=0\) on \([0,\ell]\).
\end{lemma}

\begin{proof}
Set $f(y)=G(\ell-y)$ for $0<y<\ell$, so that $f\in L^1(0,\ell)$.
The change of variables $r=\ell-y$ and~\eqref{eq:discrete-moment-assumption} give
\[
  \begin{aligned}
    \int_0^\ell f(y)e^{-Ny}\,\mathrm{d} y
    &=e^{-N\ell}\int_0^\ell G(r)e^{Nr}\,\mathrm{d} r\\
    &=O(N^{-1}e^{-N\ell})
    \qquad(N\to\infty,\ N\in\mathbb{N}).
  \end{aligned}
\]
For every $\varepsilon>0$, this is $O(e^{-\ell N(1-\varepsilon)})$.
Thus~\cite[Proposition~2.4]{horvath2011partial}, with $a=\ell$, gives $f=0$ almost everywhere.
The continuity of $G$ then implies $G=0$ on $[0,\ell]$.
\end{proof}

The next lemma shows that the values of the Weyl function at sufficiently large integers determine the potential.

\begin{lemma}\label{lem:discrete-weyl}
Let $S>0$ and $V_1,V_2\in C^\infty([0,S];\mathbb{R})$.
For $j=1,2$, let $\mathfrak m_j$ be the Weyl function for $V_j$ with the right Dirichlet condition $v(S)=0$, as defined in Lemma~\ref{lem:a-amplitude}.
Assume that, for all sufficiently large integers $N$,
\[
  \mathfrak m_1(N)=\mathfrak m_2(N).
\]
Then
\[
  V_1=V_2\qquad\text{on }[0,S].
\]
\end{lemma}

\begin{proof}
Let
$\delta\mathfrak a=\mathfrak a_{V_1}-\mathfrak a_{V_2}$
be the difference of the two $A$-amplitudes.
Fix $0<a<S$ and choose $b\in(a,S)$.
Taking $\kappa=N$ and $\varepsilon=b-a$ in~\eqref{eq:80} and using $\mathfrak m_1(N)=\mathfrak m_2(N)$, we obtain
\[
  \int_0^b e^{-2N\alpha}\delta\mathfrak a(\alpha)\,\mathrm{d}\alpha
  =O(e^{-(a+b)N})
\]
as $N\to\infty$ through the integers.

Since $\delta\mathfrak a$ is continuous on $[0,b]$,
\[
  N\left|\int_a^b e^{-2N\alpha}\delta\mathfrak a(\alpha)\,\mathrm{d}\alpha\right|
  \le\frac12\|\delta\mathfrak a\|_{L^\infty(a,b)}e^{-2Na}.
\]
Also, $Ne^{-(a+b)N}=O(e^{-2Na})$ because $b>a$.
Subtracting the integral over $[a,b]$ therefore yields
\begin{equation}
  N\int_0^a e^{-2N\alpha}\delta\mathfrak a(\alpha)\,\mathrm{d}\alpha
  =O(e^{-2Na}).
  \label{eq:81}
\end{equation}
Set
$G(r)=\delta\mathfrak a(a-r/2),\qquad 0\le r\le2a.$
Then $G\in C([0,2a])$, and the change of variables $r=2(a-\alpha)$ turns~\eqref{eq:81} into
\begin{equation}
  \sup_{N\ge N_0}
  N\left|\int_0^{2a}e^{Nr}G(r)\,\mathrm{d} r\right|<\infty,
  \qquad N\in\mathbb{N},
  \label{eq:82}
\end{equation}
for some sufficiently large integer $N_0$.
Lemma~\ref{lem:discrete-exponential-moments}, with $\ell=2a$, gives $G=0$ on $[0,2a]$, and hence $\mathfrak a_{V_1}=\mathfrak a_{V_2}$ on $[0,a]$.
Lemma~\ref{lem:a-amplitude} now gives $V_1=V_2$ on $[0,a]$.
Since $a<S$ was arbitrary, continuity gives $V_1=V_2$ on $[0,S]$.
\end{proof}

We next recover the potential and the Liouville factor for each Fourier direction.

\begin{lemma}\label{lem:directional-liouville-factor}
Let $g_1,g_2$ belong to the tangentially homogeneous metric class $\mathcal S(M)$ of Definition~\ref{def:tangentially-homogeneous-class}.
Assume that their DtN maps satisfy $\Lambda_{g_1}=\Lambda_{g_2}$.
Then, for every $\xi\in\mathbb{Z}^d\setminus\{0\}$, one has
\[
  V_{1,\xi}=V_{2,\xi},\qquad
  \mu_{1,\xi}=\mu_{2,\xi}
  \qquad\text{on }[0,S_\xi],
\]
where $S_\xi$ is the common interval length in~\eqref{eq:79}.
\end{lemma}

\begin{proof}
Fix $\xi\in\mathbb{Z}^d\setminus\{0\}$.
The interval lengths agree by~\eqref{eq:79}, and the Weyl samples agree by~\eqref{eq:76}.
Lemma~\ref{lem:discrete-weyl} therefore gives $V_{1,\xi}=V_{2,\xi}$ on $[0,S_\xi]$.
By~\eqref{eq:68}, both Liouville factors satisfy
$\mu_{j,\xi}''=V_{1,\xi}\mu_{j,\xi},
  \qquad 0<s<S_\xi,\qquad j=1,2.$
Proposition~\ref{prop:directional-weyl-data} also gives the common initial data
$\mu_{1,\xi}(0)=\mu_{2,\xi}(0),\qquad
  \mu_{1,\xi}'(0)=\mu_{2,\xi}'(0).$
Uniqueness for this initial value problem~\cite[Theorem~3.9]{teschl2012ordinary} therefore yields
\begin{equation}
  \mu_{1,\xi}(s)=\mu_{2,\xi}(s)=:\mu_\xi(s)
  \qquad(0\le s\le S_\xi).
  \label{eq:83}
\end{equation}
\end{proof}

\subsection{Metric recovery}\label{subsec:homogeneous-metric-recovery}

Subsection~\ref{subsec:homogeneous-fourier} extracts the Weyl samples from the DtN map, and Subsection~\ref{subsec:homogeneous-inverse-spectral} recovers the corresponding potentials and Liouville factors.
Consequently, Lemma~\ref{lem:directional-liouville-factor} gives, for each nonzero lattice direction $\xi$, a common Liouville interval $[0,S_\xi]$ and a common positive factor $\mu_\xi(s)$.
This does not yet recover $h_1$ and $h_2$ in the prescribed normal coordinate $t$, because the change of variable $s_{j,\xi}(t)$ depends on both the metric and $\xi$.
The identity $\mu_{1,\xi}(s)=\mu_{2,\xi}(s)$ may therefore compare the metrics at different values of $t$, and the correspondence between these values may also depend on $\xi$.
The following proposition removes these ambiguities by constructing a common parameter in which all directions can be compared, recovering the cofactor matrices $\operatorname{cof}(h_j)$ from finitely many directions, and proving that the two changes from the common parameter back to $t$ coincide.

\begin{proposition}\label{prop:liouville-synchronization}
Let $d\ge2$, and let $g_1,g_2$ belong to the tangentially homogeneous metric class $\mathcal S(M)$ of Definition~\ref{def:tangentially-homogeneous-class}.
Assume that their DtN maps satisfy $\Lambda_{g_1}=\Lambda_{g_2}$.
Then the tangential coefficient matrices $h_1,h_2$ in~\eqref{eq:15} satisfy
\[
  h_1(t)=h_2(t)\qquad(0\le t\le L).
\]
\end{proposition}

\begin{proof}
For each $\xi\in\mathbb{Z}^d\setminus\{0\}$, Proposition~\ref{prop:directional-weyl-data} and Lemma~\ref{lem:directional-liouville-factor} give the common interval length $S_\xi$ and positive Liouville factor $\mu_\xi$ in~\eqref{eq:79} and~\eqref{eq:83}.
Recall that $\omega_j(t)=\sqrt{\det h_j(t)}$, $q_{j,\xi}(t)=\xi^Th_j(t)^{-1}\xi$, and $s_{j,\xi}'(t)=\sqrt{q_{j,\xi}(t)}$ by~\eqref{eq:66}.
The identity~\eqref{eq:67}, together with~\eqref{eq:83}, gives $\mu_\xi(s_{j,\xi}(t))^2=\omega_j(t)\sqrt{q_{j,\xi}(t)}$.
These two relations suggest a parameter in which the dependence on $\xi$ cancels.
For $j=1,2$, define
\[
  \tau_j(t)=\int_0^t\omega_j(r)^{-1}\,\mathrm{d}r,
  \qquad 0\le t\le L.
\]
This parameter depends on the metric but not on the direction.
For each $\xi\in\mathbb{Z}^d\setminus\{0\}$, define from the recovered factor
\[
  \tau_\xi(s)=\int_0^s\mu_\xi(r)^{-2}\,\mathrm{d}r,
  \qquad 0\le s\le S_\xi.
\]
The chain rule and the preceding identities give, for $j=1,2$,
\begin{equation}
  \begin{aligned}
  \frac{\mathrm{d}}{\mathrm{d}t}\tau_\xi(s_{j,\xi}(t))
  &=\tau_\xi'(s_{j,\xi}(t))s_{j,\xi}'(t)\\
  &=\frac{\sqrt{q_{j,\xi}(t)}}{\mu_\xi(s_{j,\xi}(t))^2}
   =\omega_j(t)^{-1}
   =\tau_j'(t).
  \end{aligned}
  \label{eq:84}
\end{equation}
Since $\tau_\xi(s_{j,\xi}(0))=\tau_j(0)=0$, integration gives
$\tau_\xi(s_{j,\xi}(t))=\tau_j(t),
  \qquad 0\le t\le L.$
At $t=L$, the common values of $S_\xi$ and $\mu_\xi$ show that the right-hand side has the same endpoint for $j=1,2$.
Since $\tau_j(L)$ is independent of $\xi$, this endpoint is also the same for every direction; denote it by $\mathcal T$:
\[
  \tau_1(L)=\tau_2(L)=\tau_\xi(S_\xi)=:\mathcal T,
  \qquad \xi\in\mathbb{Z}^d\setminus\{0\}.
\]
All the maps $\tau_j$ and $\tau_\xi$ are smooth and strictly increasing.
From now on, $\tau\in[0,\mathcal T]$ denotes a scalar variable in their common range, rather than an additional map.
For each such $\tau$, define
$t_j(\tau):=\tau_j^{-1}(\tau),
  \qquad
  s_\xi(\tau):=\tau_\xi^{-1}(\tau).$
Substituting $t=t_j(\tau)$ into the preceding identity gives
$\tau_\xi(s_{j,\xi}(t_j(\tau)))
  =\tau_j(t_j(\tau))
  =\tau.$
Since $\tau_\xi$ is injective and $\tau_\xi(s_\xi(\tau))=\tau$, it follows that
$s_{j,\xi}(t_j(\tau))=s_\xi(\tau),
  \qquad 0\le\tau\le\mathcal T,
  \quad j=1,2.$
Thus the two metrics and all lattice directions are now compared at the same value of $\tau$.

We next recover the metric in this common parameter.
For a positive definite matrix $h$, write $\operatorname{cof}(h)=(\det h)h^{-1}$.
Set
$\mathsf C_j(\tau)
  :=\operatorname{cof}(h_j(t_j(\tau))),
  \qquad j=1,2.$
Using~\eqref{eq:67} at $t=t_j(\tau)$ and the synchronization just proved, we obtain
\begin{equation}
  \begin{aligned}
    \mu_\xi(s_\xi(\tau))^4
    &=\omega_j(t_j(\tau))^2q_{j,\xi}(t_j(\tau))\\
    &=\xi^T\mathsf C_j(\tau)\xi,
    \qquad j=1,2.
  \end{aligned}
  \label{eq:85}
\end{equation}
For each fixed $\tau$, the left-hand side is independent of $j$, so $\mathsf C_1(\tau)$ and $\mathsf C_2(\tau)$ have the same quadratic form in every nonzero lattice direction.
It is enough to use the finite set introduced in the proof of Proposition~\ref{prop:directional-weyl-data}, namely
\[
  \mathcal K
  =\{e_a:1\le a\le d\}
  \cup\{e_a+e_b:1\le a<b\le d\}.
\]
The values in the directions $e_a$ determine the diagonal entries of $\mathsf C_j(\tau)$, while
\[
  2(\mathsf C_j(\tau))_{ab}
  =(e_a+e_b)^T\mathsf C_j(\tau)(e_a+e_b)
   -e_a^T\mathsf C_j(\tau)e_a
   -e_b^T\mathsf C_j(\tau)e_b
\]
determines every off-diagonal entry.
Equation~\eqref{eq:85} therefore gives
$\mathsf C_1(\tau)=\mathsf C_2(\tau),
  \qquad 0\le\tau\le\mathcal T.$
Since $d\ge2$, a positive definite matrix $h$ is recovered from its cofactor matrix by
\begin{equation}
  \det(\operatorname{cof}h)=(\det h)^{d-1},
  \qquad
  h=\bigl(\det(\operatorname{cof}h)\bigr)^{1/(d-1)}
  (\operatorname{cof}h)^{-1}.
  \label{eq:86}
\end{equation}
Consequently, the equality of the cofactor matrices gives
$h_1(t_1(\tau))=h_2(t_2(\tau)).$
Using $\omega_j=\sqrt{\det h_j}$, it also gives
$\omega_1(t_1(\tau))=\omega_2(t_2(\tau))=:\omega(\tau).$
Thus the metric has been recovered in the common parameter $\tau$, but it remains to identify this parameter with the same value of the prescribed normal coordinate for both metrics.
Since $t_j$ is the inverse of $\tau_j$,
\[
  t_j'(\tau)
  =\frac{1}{\tau_j'(t_j(\tau))}
  =\omega_j(t_j(\tau))
  =\omega(\tau),
  \qquad t_j(0)=0.
\]
The two inverse changes of variable therefore solve the same initial value problem, and hence
\[
  t_1(\tau)=t_2(\tau)=\int_0^\tau\omega(r)\,\mathrm{d}r,
  \qquad 0\le\tau\le\mathcal T.
\]
Their common value maps $[0,\mathcal T]$ onto $[0,L]$.
Substituting $t=t_1(\tau)=t_2(\tau)$ into the recovered metric identity now gives
$h_1(t)=h_2(t),
  \qquad 0\le t\le L.$

The proof is complete.
\end{proof}

\begin{proof}[Proof of Theorem~\ref{thm:tangentially-homogeneous}]
The identity~\eqref{eq:tangentially-homogeneous-data}, together with tangential homogeneity and Lemma~\ref{lem:fourier-density-flux}, gives equality of the corresponding Fourier blocks of the two DtN maps.
Proposition~\ref{prop:directional-weyl-data} recovers the common endpoint densities, and hence the weighted endpoint flux matrices $D_j(k)$ in~\eqref{eq:65} agree; for every $\xi\in\mathcal K$, it also gives common Liouville interval lengths, common initial values of the Liouville factors, and equal Weyl samples.
Lemma~\ref{lem:discrete-weyl} recovers the corresponding one-dimensional potentials, and Lemma~\ref{lem:directional-liouville-factor} then recovers the Liouville factors.
Proposition~\ref{prop:liouville-synchronization} synchronizes the Liouville coordinates, recovers the tangential coefficient matrices by polarization, and gives $h_1(t)=h_2(t)$ for all $t\in[0,L]$.
Together with the prescribed metric form~\eqref{eq:15}, this proves $g_1=g_2$ on $M$, as asserted in~\eqref{eq:tangentially-homogeneous-conclusion}.
\end{proof}

\section{Smooth partial-data uniqueness under layerwise Loewner ordering}\label{sec:proof-loewner}

This section is devoted to the proofs of Theorem~\ref{thm:loewner} and Corollary~\ref{cor:loewner-metric}.
Let $\Omega$, $\Gamma$, $\rho$, and $\gamma_1,\gamma_2$ satisfy the hypotheses of Theorem~\ref{thm:loewner}, and write $A=\gamma_1-\gamma_2$.

The conductivity proof has three steps.
First, monotonicity converts equality of the partial DtN maps into integral inequalities involving $A$.
Second, partial-data Runge approximation and local harmonic functions with prescribed gradients produce localized potentials that detect nonzero directions of $A$ and yield uniqueness in an exterior layer under a one-sided Loewner ordering.
Third, the layerwise ordering and geometric assumptions allow continuous layer stripping to recover the conductivities throughout $\Omega$.
The final subsection derives the metric corollary from the conductivity result.

\subsection{Monotonicity and Runge approximation}

Throughout this section, boundary values and solutions are real-valued.
We define $H_{00}^{1/2}(\Gamma)$ to be the closure in $H^{1/2}(\partial\Omega)$ of the zero extensions of $C_c^\infty(\Gamma)$ to $\partial\Omega$.
This closed subspace carries the inherited norm, and $\bigl(H_{00}^{1/2}(\Gamma)\bigr)^*$ denotes its continuous dual.
For a uniformly positive $\gamma\in C^\infty(\overline\Omega;\operatorname{Sym}_n^+)$ and $f\in H_{00}^{1/2}(\Gamma)$, let $u_\gamma^f\in H^1(\Omega)$ be the unique weak solution of
\[
  -\operatorname{div}(\gamma\nabla u_\gamma^f)=0\quad\text{in }\Omega,
  \qquad u_\gamma^f|_{\partial\Omega}=f.
\]
The partial conductivity DtN map is the bounded operator
\[
  \Lambda_\gamma^\Gamma:
  H_{00}^{1/2}(\Gamma)\longrightarrow
  \bigl(H_{00}^{1/2}(\Gamma)\bigr)^*
\]
defined by
\begin{equation}
  \langle\Lambda_\gamma^\Gamma f,h\rangle
  =\int_\Omega\gamma\nabla u_\gamma^f\cdot\nabla u_\gamma^h\,\mathrm{d}x,
  \qquad f,h\in H_{00}^{1/2}(\Gamma).
  \label{eq:loewner-local-form}
\end{equation}

By density, equality of the partial DtN maps on $C_c^\infty(\Gamma)$ is equivalent to equality of these bounded operators.

\begin{lemma}\label{lem:loewner-monotonicity}
Let $f\in H_{00}^{1/2}(\Gamma)$ and write $u_j^f=u_{\gamma_j}^f$ for the $\gamma_j$-harmonic solution with trace $f$.
With $A=\gamma_1-\gamma_2$, one has
\begin{equation}
  \int_\Omega A\nabla u_1^f\cdot\nabla u_1^f\,\mathrm{d} x
  \le
  \langle(\Lambda_{\gamma_1}^\Gamma-\Lambda_{\gamma_2}^\Gamma)f,f\rangle
  \le
  \int_\Omega A\nabla u_2^f\cdot\nabla u_2^f\,\mathrm{d} x.
  \label{eq:loewner-monotonicity}
\end{equation}
\end{lemma}

\begin{proof}
We use the standard variational argument based on the Dirichlet principle; see~\cite{harrach2013monotonicity,garde2025direct} for the corresponding Neumann-to-Dirichlet estimates.
For $j=1,2$, put
\[
  E_j(v)=\int_\Omega\gamma_j\nabla v\cdot\nabla v\,\mathrm{d} x.
\]
The solution $u_j^f$ minimizes $E_j$ among functions with trace $f$.
Consequently,
\begin{align*}
  E_1(u_1^f)-E_2(u_2^f)
  &=E_2(u_1^f)-E_2(u_2^f)
    +\int_\Omega A\nabla u_1^f\cdot\nabla u_1^f\,\mathrm{d} x,\\
  E_1(u_1^f)-E_2(u_2^f)
  &=E_1(u_1^f)-E_1(u_2^f)
    +\int_\Omega A\nabla u_2^f\cdot\nabla u_2^f\,\mathrm{d} x.
\end{align*}
The first energy difference on the first line is nonnegative, whereas the first energy difference on the second line is nonpositive.
Equation~\eqref{eq:loewner-local-form}, with $\gamma=\gamma_j$ and $h=f$, gives
$\langle\Lambda_{\gamma_j}^\Gamma f,f\rangle=E_j(u_j^f).$
Thus the middle term in \eqref{eq:loewner-monotonicity} equals $E_1(u_1^f)-E_2(u_2^f)$, and the two displayed decompositions give the asserted inequalities.
\end{proof}

We next state the partial-data Runge approximation in conductivity form.

\begin{lemma}\label{lem:loewner-runge}
Let $\gamma\in C^\infty(\overline\Omega;\operatorname{Sym}_n^+)$ be uniformly positive.
Let $U\Subset\Omega$ be a smooth, possibly disconnected, open set such that $\Omega\setminus\overline U$ is connected, and let $v\in H^1(U)$ satisfy $-\operatorname{div}(\gamma\nabla v)=0$ in $U$.
Then there exists a sequence $(f_\ell)_{\ell=1}^\infty\subset H_{00}^{1/2}(\Gamma)$ such that the corresponding solutions $u_\gamma^{f_\ell}$ satisfy
\[
  u_\gamma^{f_\ell}\longrightarrow v
  \quad\text{in }L^2(U)
  \quad\text{as }\ell\to\infty.
\]
Moreover, the same sequence converges to $v$ in $H^1(U_0)$ for every $U_0\Subset U$.
\end{lemma}

\begin{proof}
The duality and unique-continuation argument underlying Lemma~\ref{lem:runge-approximation}, applied to $-\operatorname{div}(\gamma\nabla)$, gives the $L^2(U)$ approximation~\cite[Theorem~1 and Section~2]{ruland2019quantitative}.
For $z\in L^2(U)$ orthogonal to all reachable restrictions, the zero-Dirichlet adjoint solution with the zero extension of $z$ as source has zero conormal trace on $\Gamma$ by Green's identity and hence vanishes in the connected set $\Omega\setminus\overline U$ by unique continuation.
A second application of Green's identity on each component of $U$ shows that $z$ is orthogonal to every local solution, yielding density by Hahn--Banach; for Lemma~\ref{lem:runge-approximation}, the same argument uses Riemannian volume and boundary measures.
The interior Caccioppoli estimate yields convergence in $H^1(U_0)$ for every $U_0\Subset U$.
\end{proof}

\subsection{Localized potentials and layer recovery}

We next use the method of localized potentials, which has also been employed in~\cite{gebauer2008localized,harrach2017local,garde2025direct}.

\begin{lemma}\label{lem:loewner-localized-potentials}
Let $\gamma\in C^\infty(\overline\Omega;\operatorname{Sym}_n^+)$ be uniformly positive.
Let $D\Subset\Omega$ be a smooth open set with $\Omega\setminus\overline D$ connected, and fix $x_0\in\Omega\setminus\overline D$.
Let $B_0^+\Subset\Omega\setminus\overline D$ be a ball centered at $x_0$, and let $v$ be $\gamma$-harmonic on $B_0^+$.
Then there exist balls $B\Subset B^+\Subset B_0^+$ centered at $x_0$ and a sequence $(f_m)_{m=1}^\infty\subset H_{00}^{1/2}(\Gamma)$ such that the corresponding solutions $u_m=u_\gamma^{f_m}$ satisfy
\begin{equation}
  \|u_m-mv\|_{H^1(B)}\le m^{-1},
  \qquad
  \|u_m\|_{H^1(D)}\le m^{-1}.
  \label{eq:localized-potential-estimates}
\end{equation}
\end{lemma}

\begin{proof}
Choose a sufficiently small outward normal enlargement $D^+$ satisfying
\begin{equation}
  \overline D\subset D^+\Subset\Omega,
  \qquad x_0\notin\overline{D^+},
  \qquad \Omega\setminus\overline{D^+}\text{ is connected}.
  \label{eq:localized-enlargement}
\end{equation}
Such an enlargement is obtained by an ambient tubular flow, which preserves connectedness of the exterior.

Choose balls $B\Subset B^+\Subset B_0^+$ centered at $x_0$ such that $2B^+\Subset\Omega\setminus\overline{D^+}$.
Then
\begin{equation}
  \Omega\setminus(\overline{D^+}\cup\overline{B^+})
  \quad\text{connected}.
  \label{eq:localized-complement}
\end{equation}
Indeed, paths in the exterior of $D^+$ can be rerouted around $\overline{B^+}$ through the annulus $2B^+\setminus\overline{B^+}$, which is connected for $n\ge2$.

Set $U=D^+\cup B^+$ and, for each $m\ge1$, define
\[
  w_m=0\quad\text{on }D^+,
  \qquad
  w_m=mv\quad\text{on }B^+.
\]
The function $w_m$ is $\gamma$-harmonic on $U$.
Since $D\cup B\Subset U$ and \eqref{eq:localized-complement} holds, Lemma~\ref{lem:loewner-runge} gives $f_m\in H_{00}^{1/2}(\Gamma)$ such that
$\|u_\gamma^{f_m}-w_m\|_{H^1(D\cup B)}\le m^{-1}.$
Thus $u_m=u_\gamma^{f_m}$ satisfies \eqref{eq:localized-potential-estimates}.
\end{proof}

The following proposition gives uniqueness in an exterior region under a one-sided Loewner ordering.

\begin{proposition}\label{prop:loewner-exterior}
Let $D\Subset\Omega$ be a smooth open set such that $\Omega\setminus\overline D$ is connected.
Set $A=\gamma_1-\gamma_2$.
Assume that the partial DtN maps agree as in~\eqref{eq:loewner-data}, and that either $A\succeq0$ throughout $\Omega\setminus D$ or $A\preceq0$ throughout $\Omega\setminus D$.
Then
\begin{equation}
  A=0\qquad\text{on }\Omega\setminus D.
  \label{eq:exterior-one-sided-zero}
\end{equation}
\end{proposition}

\begin{proof}
Assume first that $A\succeq0$ on $\Omega\setminus D$ and that \eqref{eq:exterior-one-sided-zero} fails.
By continuity, there are $x_0\in\Omega\setminus\overline D$ and $\xi\in\mathbb{R}^n$ such that
$\xi^TA(x_0)\xi>0.$
For $n\ge3$, the local harmonic-coordinate argument in the proof of Lemma~\ref{lem:runge-jets}\textup{(R2)}, applied to the metric associated with $\gamma_1$ by~\eqref{eq:metric-from-density}, gives a local $\gamma_1$-harmonic function $v$ with $\nabla v(x_0)=\xi$.
For $n=2$, we instead solve the $\gamma_1$-harmonic Dirichlet problems on a sufficiently small ball centered at $x_0$ with the two coordinate functions as boundary data; rescaling and interior elliptic estimates give linearly independent gradients at $x_0$, so a linear combination yields the required $v$.

By continuity, choose a ball $B_0^+\Subset\Omega\setminus\overline D$ centered at $x_0$, on which $v$ is $\gamma_1$-harmonic and
$A(x)\nabla v(x)\cdot\nabla v(x)>0\qquad(x\in B_0^+).$
Lemma~\ref{lem:loewner-localized-potentials} gives balls $B\Subset B^+\Subset B_0^+$ and global $\gamma_1$-harmonic functions $u_m$ satisfying \eqref{eq:localized-potential-estimates}.
In particular,
\begin{equation}
  I:=\int_BA\nabla v\cdot\nabla v\,\mathrm{d} x>0.
  \label{eq:local-positive-energy}
\end{equation}
It follows from \eqref{eq:localized-potential-estimates} and \eqref{eq:local-positive-energy} that
\begin{align}
  \int_BA\nabla u_m\cdot\nabla u_m\,\mathrm{d} x&=m^2I+O(1),
  \label{eq:localized-positive}\\
  \int_DA\nabla u_m\cdot\nabla u_m\,\mathrm{d} x&\ge-Cm^{-2},
  \label{eq:localized-core}
\end{align}
where $C$ is independent of $m$.
Because $A\succeq0$ on $\Omega\setminus D$, the integrand is nonnegative on $\Omega\setminus(D\cup B)$.
Consequently,
\[
  \int_\Omega A\nabla u_m\cdot\nabla u_m\,\mathrm{d} x>0
\]
for all sufficiently large $m$.
On the other hand, \eqref{eq:loewner-data} and the first inequality in \eqref{eq:loewner-monotonicity} give
\[
  \int_\Omega A\nabla u_m\cdot\nabla u_m\,\mathrm{d} x\le0,
\]
a contradiction.
This proves \eqref{eq:exterior-one-sided-zero} for $A\succeq0$; equality on $\partial D$ follows by continuity.

If $A\preceq0$ on $\Omega\setminus D$, repeat the local construction with $\gamma_2$ and $-A$, and apply Lemma~\ref{lem:loewner-localized-potentials}.
The resulting solutions make
\[
  \int_\Omega A\nabla u_m\cdot\nabla u_m\,\mathrm{d} x<0
\]
for large $m$, whereas \eqref{eq:loewner-data} and the second inequality in \eqref{eq:loewner-monotonicity} require this integral to be nonnegative.
This proves the negative-semidefinite case.
\end{proof}

\subsection{Continuous layer stripping}

We now propagate the exterior recovery of Proposition~\ref{prop:loewner-exterior} through the level sets of $\rho$.

\begin{proof}[Proof of Theorem~\ref{thm:loewner}]
Put $A=\gamma_1-\gamma_2$ and recall that $D_s=\{\rho>s\}$ for $0<s<T$.
The condition $\,\mathrm{d}\rho\ne0$ in~\eqref{eq:loewner-foliation} makes each $D_s$ smooth, while~\eqref{eq:loewner-topology} ensures that $D_s\Subset\Omega$ and its exterior is connected.
Together with the equality of partial DtN maps in~\eqref{eq:loewner-data}, these properties allow us to apply Proposition~\ref{prop:loewner-exterior} whenever $A$ is positive semidefinite or negative semidefinite throughout $\{\rho\le s\}$.

We first recover a boundary layer.
The one-sided ordering condition~\eqref{eq:loewner-order}, applied at $s=0$, gives $b_0\in(0,T]$ and $\sigma_0\in\{1,-1\}$ such that $\sigma_0A\succeq0$ on $\{0<\rho<b_0\}$.
Choose $b\in(0,b_0)$.
Since $\rho^{-1}(0)=\partial\Omega$ by~\eqref{eq:loewner-foliation}, continuity extends this sign condition to $\{\rho\le b\}$.
Proposition~\ref{prop:loewner-exterior}, applied to $D_b$, therefore gives
\begin{equation}
  A=0\qquad\text{on }\{\rho\le b\}.
  \label{eq:initial-layer}
\end{equation}

We next show that recovery cannot stop at a level below $T$.
Define the set of recovered levels and its supremum by
\[
  \mathcal R=\{a\in[0,T):A=0\text{ on }\{\rho\le a\}\},
  \qquad
  \tau=\sup\mathcal R.
\]
The initial recovery~\eqref{eq:initial-layer} gives $\tau>0$.
Suppose that $\tau<T$.
By the definition of $\tau$, one has $A=0$ on $\{\rho<\tau\}$.
The level $\{\rho=\tau\}$ is regular by~\eqref{eq:loewner-foliation}, so its points can be approached from $\{\rho<\tau\}$; continuity of $A$ therefore gives $A=0$ on $\{\rho\le\tau\}$.

The ordering condition~\eqref{eq:loewner-order} at $s=\tau$ supplies $b_\tau\in(\tau,T]$ and $\sigma_\tau\in\{1,-1\}$ such that $\sigma_\tau A\succeq0$ on $\{\tau<\rho<b_\tau\}$.
Choose $b\in(\tau,b_\tau)$.
Since $A$ already vanishes on the recovered region $\{\rho\le\tau\}$, this new sign condition holds throughout $\{\rho\le b\}$.
In particular, $\sigma_\tau$ need not agree with the signs used in earlier layers.
Proposition~\ref{prop:loewner-exterior}, applied to $D_b$, now gives $A=0$ on $\{\rho\le b\}$.
Thus $b\in\mathcal R$, contradicting $b>\tau=\sup\mathcal R$.

Finally, $\tau=T$ implies $A=0$ on $\{\rho<T\}$.
By~\eqref{eq:loewner-topology}, the set $\{\rho=T\}$ has empty interior in $\Omega$, so $\{\rho<T\}$ is dense in $\Omega$.
Continuity therefore yields $\gamma_1=\gamma_2$ throughout $\Omega$, proving~\eqref{eq:loewner-conclusion}.
\end{proof}

\subsection{Metric interpretation}

We conclude with the proof of Corollary~\ref{cor:loewner-metric}.

\begin{lemma}\label{lem:loewner-metric-interpretation}
Let $n\ge3$, let $g$ be a smooth Riemannian metric on $\overline\Omega$, and set $\gamma_g=(\det g)^{1/2}g^{-1}$.
Then $g$ is recovered pointwise from $\gamma_g$ by
\begin{equation}
  g=(\det\gamma_g)^{1/(n-2)}\gamma_g^{-1}.
  \label{eq:metric-from-density}
\end{equation}
Moreover, if $F:\overline\Omega\to\overline\Omega$ is a smooth diffeomorphism with $F|_\Gamma=\operatorname{Id}$, then
\begin{equation}
  \Lambda_{\gamma_{F^*g}}^\Gamma=\Lambda_{\gamma_g}^\Gamma.
  \label{eq:loewner-gauge-invariance}
\end{equation}
\end{lemma}

\begin{proof}
Taking determinants in $\gamma_g=(\det g)^{1/2}g^{-1}$ gives
$\det\gamma_g=(\det g)^{(n-2)/2}.$
Since $n\ge3$, we can solve for $\det g$ and substitute back to obtain~\eqref{eq:metric-from-density}.

For the DtN identity, take $f\in C_c^\infty(\Gamma)$ and let $u=u_{\gamma_g}^f$.
Since $F$ fixes $\Gamma$ pointwise and maps $\partial\Omega\setminus\Gamma$ onto itself, the boundary value extended by zero satisfies $f\circ F=f$ on $\partial\Omega$.
A change of variables shows that $v=u\circ F$ is $\gamma_{F^*g}$-harmonic with the same boundary value $f$, and that
\[
  \int_\Omega\gamma_{F^*g}\nabla v\cdot\nabla v\,\mathrm{d} x
  =\int_\Omega\gamma_g\nabla u\cdot\nabla u\,\mathrm{d} x.
\]
The energy representation~\eqref{eq:loewner-local-form}, followed by polarization and density of $C_c^\infty(\Gamma)$ in $H_{00}^{1/2}(\Gamma)$, proves~\eqref{eq:loewner-gauge-invariance}.
Only the accessible boundary $\Gamma$ needs to be fixed pointwise.
\end{proof}

\begin{proof}[Proof of Corollary~\ref{cor:loewner-metric}]
The invariance in Lemma~\ref{lem:loewner-metric-interpretation} and the assumed equality of partial conductivity DtN maps~\eqref{eq:loewner-metric-data} give
\[
  \Lambda_{\gamma_{F^*g_2}}^\Gamma
  =\Lambda_{\gamma_{g_2}}^\Gamma
  =\Lambda_{\gamma_{g_1}}^\Gamma.
\]
By hypothesis, the pair $(\gamma_{g_1},\gamma_{F^*g_2})$ belongs to the layerwise one-sided Loewner class $\mathcal L_\rho(\Omega)$.
\mbox{Theorem~\ref{thm:loewner}} therefore yields $\gamma_{g_1}=\gamma_{F^*g_2}$.
The inverse relation~\eqref{eq:metric-from-density} now gives $g_1=F^*g_2$ on $\Omega$, proving~\eqref{eq:loewner-metric-conclusion}.
The final assertion follows by taking $F=\operatorname{Id}$.
\end{proof}

\section*{Acknowledgments}

The work of Y. Jiang was supported by the Hong Kong RGC Project JRFS2627-1S06.
The work of H. Liu was supported by the Hong Kong RGC General Research Funds (projects 11311122, 12301420, and 11300821).
The authors acknowledge the use of AI tools.
All mathematical arguments and proofs in the final manuscript were checked and written by the authors.

\begingroup
\hfuzz=2pt
\bibliographystyle{plain}
\bibliography{ref14}
\endgroup

\end{document}